\documentclass[12pt]{amsart}

\usepackage[utf8]{inputenc}

\usepackage{amsmath}
\usepackage{amsthm}
\usepackage{amssymb}
\usepackage{enumitem}
\usepackage{nicefrac}
\usepackage{commath}
\usepackage{appendix}
\usepackage{cancel}

\usepackage{color}
\usepackage[usenames,dvipsnames]{xcolor}

\usepackage{todonotes}

\usepackage{xspace}

\usepackage[T1]{fontenc}
\usepackage{times}
\usepackage{hyperref}

\usepackage{graphicx}
\usepackage[capitalise,noabbrev]{cleveref}

\usepackage{thmtools}
\usepackage{thm-restate}

\declaretheorem[numberwithin=section]{lemma}
\declaretheorem[sibling=lemma]{theorem}

\declaretheorem[sibling=lemma]{problem}
\declaretheorem[sibling=lemma]{corollary}
\declaretheorem[sibling=lemma]{proposition}

\declaretheorem[sibling=lemma,style=remark]{example}

\declaretheorem[sibling=lemma,style=remark]{definition}

\newcommand{\A}{{\mathcal{A}}}

\newcommand{\E}{{\mathbb{E}}}
\newcommand{\C}{{\mathbb{C}}}
\newcommand{\R}{{\mathbb{R}}}
\newcommand{\Z}{{\mathbb{Z}}}

\newcommand{\uu}{{\mathfrak{u}}}

\newcommand{\gl}{{\mathfrak{gl}}}
\newcommand{\U}{\mathcal{U}}
\renewcommand{\P}{\mathbb{P}}
\newcommand{\V}{\mathbf{V}}
\newcommand{\W}{\mathbf{W}}

\newcommand{\Classical}{\operatorname{Classical}}
\newcommand{\Quantum}{\operatorname{Quantum}}
\newcommand{\p}{\overline{\wp}}

\newcommand{\N}{{\mathbb{N}}}
\newcommand{\gwia}{^{\star}}

\newcommand{\mult}{\operatorname{mult}}

\newcommand{\Sy}[1]{\mathfrak{S}({#1})}

\DeclareMathOperator{\GL}{GL}

\DeclareMathOperator{\Tr}{Tr}

\DeclareMathOperator{\End}{End}

\DeclareMathOperator{\id}{id} 
 
\DeclareMathOperator{\genus}{genus} 
\DeclareMathOperator{\defect}{defect} 
\DeclareMathOperator{\aex}{aex} 
\DeclareMathOperator{\cyc}{cyc}

\DeclareMathOperator{\Wg}{Wg}

\newcommand{\group}[1]{\mathrm{#1}}

\newcommand{\algebra}[1]{\mathcal{#1}}

\newcommand{\tr}{\operatorname{tr}}

\renewcommand{\i}{\mathbf{i}}

\newcommand{\Mat}{\operatorname{Mat}}

\newcommand{\ourn}{N}

\author {Beno\^\i{}t Collins}
\address{Department of Mathematics, Graduate School of Science,
Kyoto University, Kyoto 606-8502, Japan
\newline \indent 
CNRS, 
France}\email{collins@math.kyoto-u.ac.jp}

\author{Jonathan Novak}
\address{Department of Mathematics, UC San Diego, 
9500 Gilman Drive, La Jolla CA 92093-0112, USA}
\email{jinovak@ucsd.edu}

\author{Piotr \'Sniady}
\address{
Institute of Mathematics, Polish Academy of Sciences,
\mbox{ul.~\'Sniadec\-kich 8,} \linebreak 00-956 Warszawa, Poland
} 
\email{psniady@impan.pl}

\title[Asymptotics of Quantum Random Matrices]{Semiclassical Asymptotics \\ of  
$\mathbf{GL}_N(\C)$ Tensor Products \\ and Quantum Random Matrices}



\begin{document}

\begin{abstract}
The \emph{Littlewood--Richardson process} is a discrete random point process 
arising from the isotypic decomposition of tensor products
of irreducible representations of the linear group $\group{GL}_N(\C)$.  
\emph{Biane--Perelomov--Popov matrices} are quantum random matrices obtained
as the geometric quantization of random Hermitian matrices with deterministic eigenvalues 
and uniformly random eigenvectors.
As first observed by Biane, the correlation functions of
certain global observables of the LR process coincide with the correlation functions of 
linear statistics of sums of classically independent BPP matrices, thereby enabling
a random matrix approach to the statistical study of $\group{GL}_N(\C)$ tensor products.
In this paper, we prove an optimal result: classically independent BPP matrices
become freely independent in any semiclassical/large-dimension limit.  
This proves and generalizes a conjecture of Bufetov and Gorin, and leads
to a Law of Large Numbers for the BPP observables of the LR
process which holds in any and all semiclassical scalings.
\end{abstract}

\maketitle

\hfill \begin{minipage}{0.6\textwidth}
\emph{
To Philippe Biane,
for his 55th birthday.
}
\end{minipage}

\newcommand{\aaa}{\mathfrak{a}}
\newcommand{\bbb}{\mathfrak{b}}

\newcommand{\lattice}{\hbar_N \Z}

\section{Introduction}

\subsection{The Littlewood--Richardson process}
Rational representations of the complex general linear group, $\GL_N(\C)$, 
were classified by Schur more than a century ago, see e.g.~Weyl's classic book \cite{Weyl}.  
This classification may be stated as follows: irreducible 
representations are parametrized, 
up to isomorphsim, by configurations of $N$ hard particles on the 
one-dimensional lattice $\lattice$.  Here $\hbar_N > 0$ is an arbitrary lattice constant specifying
the regular spacing between adjacent sites.  Once Schur's classification is known, 
one may ask which particle configurations occur,
and with what multiplicity, as the signature of an irreducible component of a representation
constructed from irreducibles by means of standard operations.  In this paper, we focus on 
tensor products.  

Given a sequence
	\begin{equation}
	\label{eqn:PlanckSequence}
		\hbar_1,\hbar_2,\hbar_3,\dots 
	\end{equation} 
of positive real numbers, and two triangular arrays
	\begin{equation}
		\label{eqn:TriangularArrays}
		\begin{matrix}
				a_1^{(1)} & {} & {} & {} \\
				a_1^{(2)} & a_2^{(2)} & {} & {} \\
				a_1^{(3)} & a_2^{(3)} & a_3^{(3)} & {} \\
				\vdots & \vdots & \vdots & \ddots
			\end{matrix}	
			\quad\text{ and }\quad
			\begin{matrix}
				b_1^{(1)} & {} & {} & {} \\
				b_1^{(2)} & b_2^{(2)} & {} & {} \\
				b_1^{(3)} & b_2^{(3)} & b_3^{(3)} & {} \\
				\vdots & \vdots & \vdots & \ddots
			\end{matrix}
	\end{equation} 
such that, for each $N \in \N^*$,
	\begin{equation}
	\label{eqn:Eigenvalues}
		a_1^{(N)} > \dots > a_N^{(N)} \quad\text{ and }\quad
		b_1^{(N)} > \dots > b_N^{(N)}
	\end{equation}
specify a pair of particle configurations on $\lattice$,
let $\V_N$ and $\W_N$ be the corresponding irreducible representations
of $\GL_N(\C)$, and let 
	\begin{equation}
		\label{eqn:IsotypicDecomposition}
		\V_N \otimes \W_N = \bigoplus_{\{c_1 > \dots > c_N \} \subset 
		\lattice} \mult_N(c_1,\dots,c_N)\ \mathbf{X}^{(c_1,\dots,c_N)}
	\end{equation}
be the isotypic decomposition of the representation $\V_N \otimes \W_N$.
The multiplicities arising in this decomposition are known as \emph{Littlewood--Richardson
coefficients}.   

The decision problem 
	\begin{equation}
		\label{eqn:DecisionProblem}
		\mult_N(c_1,\dots,c_N) \stackrel{?}{>} 0
	\end{equation}
is a quantum analogue of the famous Horn problem, which asks for a characterization
of the possible spectra of the sum of two Hermitian matrices with given eigenvalues.
It is a landmark theorem of Knutson and Tao --- formerly known as the 
\emph{Saturation Conjecture} --- that the quantum and classical Horn problems 
are equivalent (see \cite{KnutsonTao} for a precise statement).  
This implies that \eqref{eqn:DecisionProblem} can
be decided in polynomial time, and a polynomial time decision algorithm has been 
given by B\"urgisser and Ikenmeyer, see \cite{BurgisserIkenmeyer2013} and references therein.  
The actual computation of 
Littlewood--Richardson coefficients is an a priori harder problem, which has been shown
to be $\#P$-complete by Narayanan \cite{Narayanan}.  Assuming 
$P\neq NP$, this rules out the existence of explicit formulas for Littlewood--Richardson
coefficients.

In lieu of satisfactory exact formulas, one may pursue 
a statistical understanding of irreducible subrepresentations of 
$\group{GL}_N(\C)$ tensor products.
More precisely, the data \eqref{eqn:TriangularArrays} 
determines a natural sequence of probability measures $\P_N$ on particle configurations 
$\{c_1>\dots>c_N\}\subset \lattice$ obtained by trading the isotypic decomposition
\eqref{eqn:IsotypicDecomposition} for the isotypic measure
	\begin{equation}
	\label{eqn:IsotypicMeasure}
		\P_N\{c_1>\dots>c_N\}\ = 
		\frac{\mult_N(c_1,\dots,c_N) \dim \mathbf{X}^{(c_1,\dots,c_N)}}{\dim\V_N \otimes \W_N}.
	\end{equation}
One may now ask about statistical features of the \emph{Littlewood--Richardson
process}, i.e.~the random point process
	\begin{equation}
	\label{eqn:DiscreteParticleEnsemble}
		c_1^{(N)} > \dots > c_N^{(N)}
	\end{equation}
on $\lattice$ whose law is $\P_N$.
This line of investigation was opened twenty years ago by Biane \cite{Biane95},
who was among the first to realize its intimate connection with random matrix theory.  

\subsection{Random matrices and asymptotic freeness}
Biane suggested that the Littlewood--Richardson process
\eqref{eqn:DiscreteParticleEnsemble} should be viewed as quantizing the continuous 
random point process
	\begin{equation}
	\label{eqn:ContinuousParticleEnsemble}
		z_1^{(N)} \geq \dots \geq z_N^{(N)}
	\end{equation}	
of eigenvalues of the random Hermitian matrix $Z_N=X_N+Y_N$
whose summands $X_N,Y_N$ are independent, uniformly 
random $N \times N$ Hermitian matrices with eigenvalues given 
by the configurations \eqref{eqn:Eigenvalues}. 
The continuum limit of the eigenvalue ensemble \eqref{eqn:ContinuousParticleEnsemble}
is well-known to be described by Voiculescu's Free Probability Theory \cite{Voiculescu1991},
as we now recall.\footnote{In the interest of brevity, we assume basic familiarity with 
Free Probability.  See Appendix \ref{app:FreeProb} for the fundamental definitions and pointers to
the literature.}

Consider the global observables of the point process $\{z_i^{(N)}\}$ defined by
	\begin{equation*}
		\overline{p}_k^{(N)} := \overline{p}_k(z_1^{(N)},\dots,z_N^{(N)}), \quad k \in \N^*,
	\end{equation*}
where 
	\begin{equation}
		\label{eqn:NewtonPowerSums}
		\overline{p}_k(x_1,\dots,x_N) = \frac{1}{N}(x_1^k+\dots +x_N^k)
	\end{equation} 
is the (normalized) Newton power sum symmetric polynomial of degree $k$ in $N$ variables.
These \emph{``Newton observables''} are nothing but the moments
of the empirical distribution of $\{z_i^{(N)}\}$. 
Clearly, one has 
	\begin{equation}
		\label{eqn:ClassicalCorrelators}
		\langle \overline{p}_{k_1}^{(N)} \cdots \overline{p}_{k_r}^{(N)} \rangle
		= \E\left[\tr(Z_N^{k_1}) \cdots \tr(Z_N^{k_r})\right],
	\end{equation}
where $\tr=N^{-1}\Tr$ is the normalized matrix trace,
$\langle \cdot \rangle$ denotes expectation with respect to the law \eqref{eqn:IsotypicMeasure} 
of the process \eqref{eqn:ContinuousParticleEnsemble}, and $\E$ denotes
expectation with respect to the law of the random matrix $Z_N$. 
Although obvious, the formula \eqref{eqn:ClassicalCorrelators} 
is very useful: it allows one to analyze correlation functions of Newton observables
by leveraging the independence of the matrix elements of $X_N$ and $Y_N$.  
This is a version of the \emph{moment method}, a ubiquitous and powerful technique 
in random matrix theory, see e.g. \cite{AGZ,NicaSpeicher_book,Novak}.

Starting with $1$-point functions, one has  
	\begin{equation*}
		\langle \overline{p}_k^{(N)} \rangle = \sum_{W} \E\tr W,
	\end{equation*}	
where the sum is over all words $W$ of length $k$ in the letters $X_N,Y_N$.
The independence of the matrix elements of $X_N,Y_N$ can be harnessed to
effectively characterize the $N \to \infty$ asymptotics of the expected trace of
each such $W$  --- what appears in the large $N$ limit is free independence. 

	\begin{theorem}[Voiculescu \cite{Voiculescu1991}]
	\label{thm:Voiculescu}
		Suppose the sequence \eqref{eqn:PlanckSequence} and the data
		\eqref{eqn:TriangularArrays} is such that the limits
			\begin{equation*}
			x_k = \lim_{N \to \infty} \overline{p}_k(a_1^{(N)},\dots,a_N^{(N)})
			\quad\text{ and }\quad
			y_k = \lim_{N \to \infty} \overline{p}_k(b_1^{(N)},\dots,b_N^{(N)})
			\end{equation*}
		exist for each $k \in \N^*$.  Then, for any fixed $d \in \N^*$ and 
		$p,q \colon [d] \to \N$,
			\begin{equation*}
				\lim_{N \to \infty} \E\tr (X_N^{p(1)}Y_N^{q(1)} \cdots 
				X_N^{p(d)}Y_N^{q(d)}) =
				\tau(X^{p(1)}Y^{q(1)} \cdots X^{p(d)}Y^{q(d)}),
			\end{equation*}
		where $X,Y$ are free random variables in a tracial noncommutative
		probability space $(\algebra{A},\tau)$ with moment sequences
		 $(x_k)_{k=1}^{\infty}$ and $(y_k)_{k=1}^{\infty}$, respectively.
		\end{theorem}
	
Let us make some remarks concerning Theorem \ref{thm:Voiculescu}.  First, the hypothesis
that the limits $x_k$ and $y_k$ exist forces $\hbar_N = O(N^{-1})$ as 
$N \to \infty$ ---
in order for the empirical distributions of the particle systems \eqref{eqn:TriangularArrays}
to converge, the lattice spacing \eqref{eqn:PlanckSequence} must decay at least as fast as 
the number of particles grows.  Second, Theorem \ref{thm:Voiculescu} implies that 
	\begin{equation*}
		\lim_{N \to \infty} \langle \overline{p}_k^{(N)} \rangle = v_k
	\end{equation*}
for each $k \in \N^*$, where the sequence $(v_k)_{k=1}^{\infty}$ is the (additive)
free convolution of the sequences $(x_k)_{k=1}^{\infty}$ and $(y_k)_{k=1}^{\infty}$. 
Third, via the decomposition 
	\begin{equation*}
		\langle \overline{p}_{k_1}^{(N)}\overline{p}_{k_2}^{(N)} \rangle
		=\sum_{W_1,W_2} \E[\tr W_1 \tr W_2],
	\end{equation*}
where the sum is over pairs of words $W_1,W_2$ in $X_N,Y_N$ of lengths $k_1,k_2$,
respectively, one can further leverage the independence of $X_N,Y_N$ to estimate 
$2$-point functions of Newton observables in the large $N$ limit and hence
demonstrate concentration of $\overline{p}_k^{(N)}$.  In this way one obtains the following 
Law of Large Numbers for the the Newton observables of the eigenvalue ensemble 
\eqref{eqn:ContinuousParticleEnsemble}.

	\begin{theorem}[Voiculescu \cite{Voiculescu1991}]
	\label{thm:ClassicalLLN}
		Under the assumptions of Theorem \ref{thm:Voiculescu}, for each 
		$k \in \N^*$ we have $\overline{p}_k(z_1^{(N)},\dots,z_N^{(N)}) \to v_k$ 
		in probability.
	\end{theorem}
	
\subsection{Biane--Perelomov--Popov quantization}
In \cite{Biane95}, Biane made the remarkable observation that an analogue of the key 
formula \eqref{eqn:ClassicalCorrelators} holds for the LR process  provided one replaces 
both the Newton observables $\overline{p}_k^{(N)}$ and the random matrices 
$X_N,Y_N$ with their quantum counterparts.  This allows one to study the LR process using
techniques analogous to those used in random matrix theory.  

Let us describe Biane's fundamental insight in more detail.  
The first  step is to understand how to quantize the classical random Hermitian matrices 
$X_N,Y_N$.  Up to minor modifications, 
the required quantization was constructed in the 1960s by 
Perelomov and Popov \cite{PerelomovPopov1967},
see also \v{Z}elobenko \cite{Zelobenko73}.  It is as follows.
For each $N \in \N^*$, introduce two $N \times N$ matrices defined by 
	\begin{align*}
		A_N :=&  \begin{bmatrix}
			{} & \vdots & {} \\
			\cdots & \rho_N(\hbar_N e_{ij}) \otimes I_{\W_N} & \cdots \\
			{} & \vdots & {} \\
		\end{bmatrix}_{1 \leq i,j \leq N}, \\
	\intertext{and}
		B_N :=&  \begin{bmatrix}
			{} & \vdots & {} \\
			\cdots & I_{\V_N} \otimes \sigma_N(\hbar_N e_{ij}) & \cdots \\
			{} & \vdots & {} \\
		\end{bmatrix}_{1 \leq i,j \leq N},
	\end{align*}
where $\{e_{ij}\}$ are the standard generators of the universal enveloping algebra
$\U(\mathfrak{gl}_N(\C))$ and $\rho_N,\sigma_N$ are the actions of $\U(\mathfrak{gl}_N(\C))$
on $\V_N$ and $\W_N$ induced by the respective linear actions of $\group{GL}_N(\C)$ on these 
vector spaces.
The matrices $A_N, B_N$ so defined are quantum random matrices in the sense that
their entries are quantum random variables living in the 
noncommutative probability space
$(\A_N,\E)$, where $\A_N$ is the algebra
	\begin{equation*}
		\A_N := \End \V_N \otimes \End \W_N
	\end{equation*}
and $\E\colon \A_N \to \C$ is the quantum expectation functional defined by
	\begin{equation*}
		\E := \tr_{\V_N} \otimes \tr_{\W_N},
	\end{equation*}
with 
	\begin{equation*}
		\tr_{\V_N}:=\frac{1}{\operatorname{dim} \V_n} \Tr_{\V_N}
		\quad\text{ and }\quad
		\tr_{\W_N}:=\frac{1}{\operatorname{dim} \W_n} \Tr_{\W_N}
	\end{equation*}
the normalized traces on $\End \V_N$ and $\End \W_N$, respectively.
We will refer to the quantum random matrices
$A_N,B_N$ as \emph{Biane--Perelomov--Popov matrices}, or BPP matrices for short.
In \cref{app:GeometricQuantization}, we present a self-contained discussion, 
in the spirit of geometric quantization, which explains why the pair $A_N,B_N$
may be viewed as a natural quantization of $X_N,Y_N$
in line with the principles of the Kirillov--Kostant orbit method.

As shown by Perelomov and Popov,
traces of powers of $A_N$ and $B_N$ are scalar operators in 
$\A_N$ --- this is the quantum analogue of the fact that the classical random 
matrices $X_N,Y_N$  have deterministic spectra.  
In fact, Perelomov and Popov showed that, for each $k \in \N^*$, one has
	\begin{align*}
		\tr(A_N^k) &= \overline{\wp}_k(a_1^{(N)},\dots,a_N^{(N)})I_{\V_N} \otimes I_{\W_N} \\
		\tr(B_N^k) &= \overline{\wp}_k(b_1^{(N)},\dots,b_N^{(N)})I_{\V_N} \otimes I_{\W_N},
	\end{align*}
where
	\begin{equation}
	\label{eqn:BPPPowerSums}
		\overline{\wp}_k(x_1,\dots,x_N)= \frac{1}{N}\sum_{i=1}^N x_i^k
		\prod_{j \neq i} \left(1 - \frac{\hbar_N}{x_i-x_j} \right) 
	\end{equation}
is a quantum deformation of the normalized Newton power sum \eqref{eqn:NewtonPowerSums}.
These deformed power sums yield the ``right'' family of global observables of the 
Littlewood--Richardson process, 
	\begin{equation*}
		\p_k^{(N)} := \overline{\wp}_k(c_1^{(N)},\dots,c_N^{(N)}), \quad k \in \N^*,
	\end{equation*}
which we will refer to as the \emph{Biane--Perelomov--Popov
observables} of the LR process,
or \emph{BPP observables} for short.  The relationship between BPP observables
of the LR process $\{c_i^{(N)}\}$ and 
BPP matrices mirrors the relationship between Newton observables
of the eigenvalue process $\{z_i^{(N)}\}$ and random Hermitian matrices:
we have
	\begin{equation}
		\label{eqn:QuantumCorrelators}
		\langle \p_{k_1}^{(N)} \cdots \p_{k_r}^{(N)} \rangle = 
		\E\left[ \tr(C_N^{k_1}) \cdots \tr(C_N^{k_r})\right],
	\end{equation}
where $\langle \cdot \rangle$ denotes expectation with respect to the law of the 
LR process, and $\E \colon \A_N \to \C$ is the quantum expectation functional 
applied to the corresponding product of normalized traces of the quantum random matrix 
$C_N=A_N+B_N$.  This is the perfect quantum analogue of \eqref{eqn:ClassicalCorrelators}.

\subsection{The semiclassical/large-dimension limit}
The existence of the formula \eqref{eqn:QuantumCorrelators} suggests the possibility of 
a moment method analysis of the BPP observables of the LR process.  
The main obstruction to implementing this idea is the extra layer of noncommutativity 
imposed by quantization: while it is true that the matrix elements of $A_N$ and $B_N$ form 
two families of classically independent quantum random variables, 
the members of these families do not commute amongst
themselves.  Instead, the matrix elements of $A_N$ and $B_N$ are 
governed by the commutation relations
	\begin{equation*}
	\begin{split}
		\label{eqn:DeformedCommutation}
		\left[(A_N)_{ij},(A_N)_{kl}\right] 
		&= \hbar_N \left(\delta_{jk}\ (A_N)_{il} - \delta_{li}\ (A_N)_{kj}\right), \\
		\left[(B_N)_{ij},(B_N)_{kl}\right] 
		&= \hbar_N \left(\delta_{jk}\ (B_N)_{il} - \delta_{li}\ (B_N)_{kj}\right),
	\end{split}
	\end{equation*}
which are inherited from the defining relations of $\U(\mathfrak{gl}_N(\C))$.
Consequently, working with mixed moments in the entries of $A_N$ and $B_N$ is 
vastly more complicated than working with mixed moments in the entries of their 
classical counterparts, $X_N$ and $Y_N$.

Despite this obstruction, 
a glance at the commutation relations \eqref{eqn:DeformedCommutation} reveals that,
if $\hbar_N$ is small, the matrix elements of each BPP matrix exhibit approximately
classical (commutative) behavior, while the pair $A_N,B_N$
retains its quantum (noncommutative) aspect --- an instance of the 
\emph{semiclassical limit}.  Moreover, when $\hbar_N$ is 
small, the BPP symmetric functions \eqref{eqn:BPPPowerSums} are approximately
equal to the Newton symmetric functions \eqref{eqn:NewtonPowerSums}.
It is thus reasonable to hope that, in the semiclassical limit, 
moment computations with BPP matrices degenerate to moment computations with 
classical random matrices, and correlation functions of BPP observables of
the LR process degenerate to correlation functions of its Newton observables.  
This would indeed be the case in a pure
semiclassical limit where $N$ is fixed and $\hbar \to 0$ independently of $N$, a regime
which arises in the context of high-dimensional representations of a fixed general linear
group \cite{CollinsSniady2006}.  However, in the present context
we must contend with the more delicate situation where $\hbar_N \to 0$ 
as $N \to \infty$.  This is a subtle coupling of the semiclassical and 
large-dimension limits in which the decay rate of $\hbar_N$ as a function
of $N$ cannot be ignored.  

In order to avoid dealing with this difficulty, previous works 
\cite{Biane95,CollinsSniady2009} have assumed  rapid decay of $\hbar_N$ in order to 
force the semiclassical limit to occur ``before'' the large $N$ limit, and argued that the
use of this contrived technical device is not a significant conceptual weakness.  However,
recent work of Bufetov and Gorin \cite{BufetovGorin2015} has called this into question
by demonstrating that the asymptotic behaviour of Newton observables of the LR process
is unexpectedly sensitive to the decay rate of $\hbar_N$ --- in particular, the results
of \cite{Biane95,CollinsSniady2009} fail when $\hbar_N$ decays linearly in $N$.  

\subsection{Main results}
The present paper is the first to analyze the asymptotics of the quantum random matrices 
$A_N,B_N$ in an arbitrary, unconditional coupling of the semiclassical and large-dimension limits,
assuming only $\hbar_N \to 0$ as $N \to \infty$, and to obtain analogues of Voiculescu's results 
(Theorems \ref{thm:Voiculescu}and \ref{thm:ClassicalLLN} above) in this generality.  

Our first main result is the counterpart of Theorem \ref{thm:Voiculescu}:
asymptotic freeness of $A_N,B_N$ in all semiclassical/large-dimension limits.

\begin{theorem}
	\label{thm:Main}
	Suppose $\hbar_N=o(1)$ as $N \to \infty$, and the data
	\eqref{eqn:TriangularArrays} is such that the limits
		\begin{equation*}
			s_k = \lim_{N \to \infty} \p_k(a_1^{(N)},\dots,a_N^{(N)})
			\quad\text{ and }\quad
			t_k = \lim_{N \to \infty} \p_k(b_1^{(N)},\dots,b_N^{(N)})
			\end{equation*}
		exist for each $k \in \N^*$.  Then, for any fixed $d \in \N^*$ and $p,q \colon [d] \to \N$, 
		\begin{equation*}
			\lim_{N \to \infty} \E\tr(A_N^{p(1)}B_N^{q(1)} \cdots A_N^{p(d)}B_N^{q(d)})
			= \tau(A^{p(1)}B^{q(1)} \cdots A^{p(d)}B^{q(d)}),
		\end{equation*}
	where $A,B$ are free random variables in a tracial noncommutative probability 
	space $(\A,\tau)$ with moment sequences $(s_k)_{k=1}^{\infty}$ and $(t_k)_{k=1}^{\infty}$,
	respectively.
\end{theorem}

Theorems \ref{thm:Voiculescu} and Theorem \ref{thm:Main} are highly
analogous --- let us compare and contrast these results.  

Whereas the hypotheses of 
Theorem \ref{thm:Voiculescu} force $\hbar_N=O(N^{-1})$ as $N \to \infty$, Theorem \ref{thm:Main}
incorporates the much weaker condition $\hbar_N=o(1)$ as an explicit hypothesis.  
The reason for this is that, unlike the Newton observables
	\begin{equation*}
		\overline{p}_k(a_1^{(N)},\dots,a_N^{(N)})
			\quad\text{ and }\quad
		\overline{p}_k(b_1^{(N)},\dots,b_N^{(N)}),
	\end{equation*}
of the data \eqref{eqn:TriangularArrays}, the BPP observables
	\begin{equation*}
		\p_k(a_1^{(N)},\dots,a_N^{(N)})
			\quad\text{ and }\quad
		\p_k(b_1^{(N)},\dots,b_N^{(N)})
	\end{equation*}
of this data only receive contributions from particles in the configurations \eqref{eqn:Eigenvalues}
which have no left neighbour, as can be seen by inspecting the definition 
\eqref{eqn:BPPPowerSums} of the BPP symmetric functions $\p_k(x_1,\dots,x_N).$
Consequently, whereas the 
existence of the limits $x_k,y_k$ forces $\hbar_N$ to decay at a rate inversely proportional
to the number particles in the configurations \eqref{eqn:Eigenvalues}, the existence of the 
limits $s_k,t_k$ only requires that $\hbar_N$ decay at a rate inversely proportional to the 
number of clusters in these configurations.  

On the other hand, there exist sequences of particle configurations with $O(1)$ clusters, 
so one could wonder whether the asymptotic freeness of $A_N,B_N$ 
would still hold true in a more general regime of $\hbar_N =O(1)$.
As we shall see in \cref{sec:counterexample}, this is not the case
and the assumption that $\hbar_N \to 0$ is indeed necessary.

\medskip

Just as Theorem \ref{thm:Voiculescu} implies the 
convergence of Newton observables of the eigenvalue ensemble $\{z_i^{(N)}\}$ 
in expectation, Theorem \ref{thm:Main} implies the convergence of BPP observables
of the LR ensemble $\{c_i^{(N)}\}$ in expectation:
	\begin{equation*}
		\lim_{N \to \infty} \langle \p_k^{(N)} \rangle = w_k\,
	\end{equation*}
where $(w_k)_{k=1}^\infty$ is the free convolution of $(s_k)_{k=1}^{\infty}$ and $(t_k)_{k=1}^{\infty}$.
Our second main result upgrades this to convergence in probability; this gives a 
counterpart of \cref{thm:ClassicalLLN} for the Littlewood--Richardson process
which holds in any and all semiclassical/large-dimension scalings limits.

\begin{theorem}
	\label{thm:QuantumLLN}
		Under the assumptions of Theorem \ref{thm:Main}, for each 
		$k \in \N^*$ we have $\p_k(c_1^{(N)},\dots,c_N^{(N)}) \to w_k$ 
		in probability.
\end{theorem}

\subsection{Relation with previous results}
Theorems \ref{thm:Main} and \ref{thm:QuantumLLN} were obtained by Biane 
in \cite{Biane95} under the very strong assumption that $\hbar_N$ decays
superpolynomially in $N$, i.e.~that $\hbar_N = o(N^{-r})$ for each $r \in \N^*$.
In this regime, the semiclassical limit rapidly overtakes the large-dimension 
limit and Theorems \ref{thm:Main} and \ref{thm:QuantumLLN} degenerate
to Theorems \ref{thm:Voiculescu} and \ref{thm:ClassicalLLN}.  In particular,
Theorem \ref{thm:ClassicalLLN} holds verbatim when the eigenvalue process
$\{z_i^{(N)}\}$ is replaced with the Littlewood--Richardson process
$\{c_i^{(N)}\}$, a fact which may be viewed as a quantitative asymptotic version of the 
Saturation Conjecture.  Collins and \'Sniady \cite{CollinsSniady2009} subsequently showed that 
Biane's assumptions could be substantially weakened, and his
results continue to hold assuming only superlinear decay of the semiclassical parameter,
$\hbar_N=o(N^{-1})$.

More recently, motivated by certain problems in 2D statistical physics, Bufetov
and Gorin \cite{BufetovGorin2015} studied the global asymptotics of the LR
process in the scaling limit where $\hbar_N$ decays linearly in $N$.  They
showed that, in this regime where the semiclassical and large-dimension limits
are ``balanced,'' quantum phenomena survive in the limit: although 
$\overline{p}_k(c_1^{(N)},\dots,c_N^{(N)})$ converges in probability to a 
constant $u_k$, the sequence $(u_k)_{k=1}^{\infty}$ is \emph{not} the 
free convolution of $(x_k)_{k=1}^{\infty}$ and $(y_k)_{k=1}^{\infty}$.
However, Bufetov and Gorin were able to show that
Theorem \ref{thm:QuantumLLN} continues to hold in the regime $\hbar_N = \Theta(N^{-1})$, 
and even obtained a precise relationship between the sequences
$(u_k)_{k=1}^{\infty}$ and $(w_k)_{k=1}^{\infty}$ similar in spirit to the classical
Markov--Krein correspondence, see \cite{BufetovGorin2015} for further discussion and
references.  This led them to conjecture 
\cite[Conjecture 1.8]{BufetovGorin2015} that the results of Biane and Collins--\'Sniady
on the asymptotic freeness of the quantum random matrices $A_N,B_N$ 
continue to hold in the regime $\hbar_N=\Theta(N^{-1})$, a fact which would yield a more conceptual
explanation of the main findings of \cite{BufetovGorin2015}.  

Theorem \ref{thm:Main} is an optimal result which subsumes the theorems of Biane and
Collins--\'Sniady, proves the conjecture of Bufetov and Gorin, 
and simultaenously generalizes all of the above to arbitrary semiclassical/large-dimension
limits.
 
\subsection{Organization and proof strategy}
The proof of \cref{thm:Main} occupies \cref{sec:MeanValuesPlanckScale} and 
\cref{sec:MeanValueAsymptotics} below.
Our proof strategy is as follows.  Fix a particular choice of the
discrete parameters $d \in \N^*$, $p,q \colon [d] \to \N$, and let
	\begin{equation}
		\label{eqn:QuantumMixedMoment}
		\tau_N = \E\tr(A_N^{p(1)}B_N^{q(1)} \cdots A_N^{p(d)}B_N^{q(d)})
	\end{equation}
be the corresponding mixed moment of $A_N,B_N$.
Building on (and, in some cases, correcting) techniques pioneered by Biane in his second 
groundbreaking paper on asymptotic representation theory \cite{Biane98},
we demonstrate that $\tau_N$ decomposes as 
	\begin{equation}
		\label{eqn:ClassicalQuantumDecomposition}
		\tau_N = \Classical_N + \hbar_N \Quantum_N,
	\end{equation}
where $\Classical_N$ and $\Quantum_N$ are
polynomial functions of the pure moments 
	\begin{equation*}
	\label{eqn:QuantumPureMoments}
		\E\tr(A_N),\dots, \E\tr(A_N^{|p|}) \quad\text{ and }\quad
		\E\tr(B_N),\dots, \E\tr(B_N^{|q|}),
	\end{equation*}	
with $|p|=p(1)+\dots+p(d)$ the $\ell^1$-norm of $p$, and
$|q|=q(1)+\dots+q(d)$ the $\ell^1$-norm of $q$.  The classical part of $\tau_N$
is independent of the Planck constant $\hbar_N$ --- its form coincides exactly with the 
resolution of the classical random matrix mixed moment
	\begin{equation*}
	\label{eqn:ClassicalMixedMoment}
		\E\tr(X_N^{p(1)}Y_N^{q(1)} \cdots X_N^{p(d)}Y_N^{q(d)})
	\end{equation*}
as a polynomial in the pure moments
	\begin{equation*}
	\label{eqn:ClassicalPureMoments}
		\E\tr(X_N),\dots, \E\tr(X_N^{|p|}) \quad\text{ and }
		\E\tr(Y_N),\dots, \E\tr(Y_N^{|q|}).
	\end{equation*}
The quantum part of $\tau_N$, which is a polynomial in the pure moments
\eqref{eqn:QuantumPureMoments} whose coefficients are themselves 
polynomials in $\hbar_N$, is present because of the noncommutativity of the 
entries of BPP matrices.  

In order to move past previous works and 
free our analysis from contrived assumptions on the decay rate of $\hbar_N$,
we must establish unconditional control on the growth of the quantum part.  Refining
the combinatorial analysis from \cite{Biane98}, 
we demonstrate that, under the assumptions of \cref{thm:Main}, the quantum part of
$\tau_N$ remains bounded as $N \rightarrow \infty$ assuming only $\hbar_N=O(1)$.  
Thus, the classical/quantum decomposition yields the estimate
	\begin{equation*}
		\label{eqn:NotFree}
		\tau_N = \Classical_N + O(\hbar_N)
	\end{equation*}
as $N \to \infty$.
It follows that $\tau_N$ agrees with its classical component up to an error
controlled by the order of magnitude of the semiclassical parameter; in particular,
we have
	\begin{equation*}
		\label{eqn:NowFree}
		\tau_N = \Classical_N + o(1)
	\end{equation*}
whenever $\hbar_N =o(1)$ as $N \to \infty$.

The negligibility of the quantum part of $\tau_N$ in the semiclassical/large-dimension limit identifies
the classical part as the ultimate source of freeness.  Going beyond Biane's computations in 
\cite{Biane95,Biane98}, which relied on techniques of Xu \cite{Xu97} for the computation
of polynomial integrals on unitary groups,
we use the full power of modern Weingarten Calculus as developed in 
\cite{Collins2003,CollinsSniady2004,MatsumotoNovak2013,Novak2010}
to show that, for any $N \geq d$, the classical part
admits an absolutely convergent 
series expansion of the form
	\begin{equation*}
		\label{eqn:ClassicalAsymptotics}
		\Classical_N = \sum_{k=0}^{\infty} \frac{e_k(N)}{N^{2k}},
	\end{equation*}
where each $e_k(N)$ is a polynomial in the pure moments \eqref{eqn:QuantumPureMoments}
whose coefficients are universal integers enumerating certain special
 ``monotone'' paths in the Cayley graph of the 
symmetric group $\Sy{d}$, as generated by the conjugacy class of transpositions.
This is a version of the \emph{topological expansion} familiar from the context of 
classical random matrix theory, and the leading term $e_0(N)$ is exactly
the free probability limit.  In particular, our proof of \cref{thm:Main} does
not rely on prior knowledge that the classical random matrices $X_N,Y_N$ are asymptotically 
free --- rather, it demonstrates that any proof which works for $X_N,Y_N$ also works
verbatim for $A_N,B_N$, in any semiclassical/large-dimension limit.

In \cref{sec:Covariance}, we generalize our classical/quantum decomposition of 
expected traces of words in $A_N,B_N$ to expectations of products of traces.
Again, we are able to show that the quantum part of this decomposition remains 
bounded even when $\hbar_N=O(1)$, so that it can be ignored provided only
$\hbar_N \to 0$.  This implies that the variance of each of the random variables
$\p_k^{(N)}$ tends to zero as $N \to \infty$, which in turn implies Theorem 
\ref{thm:QuantumLLN}.

\section{Mean Values at the Planck Scale}
\label{sec:MeanValuesPlanckScale}
In this section, we fix a particular (but arbitrary) choice of the discrete
parameters $d,p,q$, and let $\tau_N$ denote the corresponding mixed moment 
\eqref{eqn:QuantumMixedMoment}.  We
analyze $\tau_N$ at the Planck scale, $\hbar_N=\hbar$ fixed, where quantum effects
hold full sway.  In particular, we derive the classical/quantum decomposition of
$\tau_N$ announced above as equation \eqref{eqn:ClassicalQuantumDecomposition}
below.  The results of this section are non-asymptotic, i.e.~they hold for any $N \in \N^*$.  

\subsection{Unitary invariance}
Our starting point is the following observation of Biane: unitary invariance survives 
quantization.  More precisely, we have the following distributional symmetry of $A_N$
and $B_N$.

	\begin{proposition}[{\cite[Section 9.2]{Biane98}}]
	\label{prop:Invariance}
		Let $\group{U}(N)$ denote the group of $N \times N$ complex
		unitary matrices, and define a function $f_N \colon \group{U}(N) \to \C$ by 
			\begin{equation*}
				f_N(U) := \E\tr\left(UA_N^{p(1)}U^{-1}B_N^{q(1)} \cdots 
				UA_N^{p(d)}U^{-1}B_N^{q(d)}\right).
			\end{equation*}
		Then, $f_N$ is constant, being equal to $\tau_N$ for all $U 
		\in \group{U}(N)$.	
	\end{proposition}
	
As a consequence of \cref{prop:Invariance}, we have
	\begin{equation*}
		\tau_N = \int_{\group{U}(N)} f_N(U)\ \mathrm{d}U,
	\end{equation*}
where the integration is against the unit-mass Haar measure on $\group{U}(N)$.
Expanding the trace, this averaging invariance gives us the following
representation:
	\begin{multline*}
		\tau_N=\frac{1}{N}\sum_{r\colon [4d] \to [N]} \int_{\group{U}(N)} \\
\shoveleft{		\E\left[ U_{r(1)r(2)}\left(A_N^{p(1)}\right)_{r(2)r(3)}U^{-1}_{r(3)r(4)}
		 \left(B_N^{q(1)}\right)_{r(4)r(5)} 
		\cdots \right. }\\ 
		\left. \cdots U_{r(4d-3)r(4d-2)}\left(A_N^{p(d)}\right)_{r(4d-2)r(4d-1)}
		U^{-1}_{r(4d-1)r(4d)} \left(B_N^{q(d)}\right)_{r(4d)r(1)}\right] \dif U.
	\end{multline*}

Let us reparametrize the summation index $r\colon [4d] \to [N]$ 
by the quadruple of functions $i,j,i',j'\colon [d] \to [N]$ defined by
		\begin{multline*}
			\big(r(1),r(2),r(3),r(4),\dots,r(4d-3),r(4d-2),r(4d-1),r(4d)\big) \\
			=:  \big(\underbrace{\underbrace{i(1),j(1),j'(1),i'(1)}_{\text{group of four}},\dots,
			\underbrace{i(d),j(d),j'(d),i'(d)}_{\text{group of four}}}_{\text{$d$ groups}}\big).
		\end{multline*}
Then, using the classical independence of the families of (quantum) random variables
$\{ (A_N)_{ij} \}$ and $\{ (B_N)_{ij} \}$ in $(\A_N,\E)$, the above becomes
	\begin{align*}
		\tau_N &= \frac{1}{N}\sum_{i,j,i',j'\colon [d] \to [N]} I_N(i,j,i',j') \times \\ 
                       & \hspace{10ex} \E\left[ \prod_{k=1}^d
				\left(A_N^{p(k)}\right)_{j(k)j'(k)} \left(B_N^{q(k)}\right)_{i'(k)i\gamma(k)}\right] \\
		&= \frac{1}{N}\sum_{i,j,i',j'\colon [d] \to [N]} I_N(i,j,i',j')\times \\ 
                       & \hspace{10ex} \E\left[ \prod_{k=1}^d
				\left( A_N^{p(k)}\right)_{j(k)j'(k)} \right]
			\E\left[ \prod_{k=1}^d \left(B_N^{q(k)}\right)_{i'(k)i\gamma(k)} \right],
	\end{align*}
where 
\[\gamma:=(1\ 2\ \dots\ d)\] 
is the full forward cycle in the symmetric group $\Sy{d}$, and 
	\begin{equation}
		\label{eqn:MatrixIntegral}
		I_N(i,j,i',j') := \int_{\group{U}(N)} 
			\prod_{k=1}^dU_{i(k)j(k)} \overline{U}_{i'(k)j'(k)} \dif U.
	\end{equation}

\subsection{The Weingarten function}
Matrix integrals of the form \eqref{eqn:MatrixIntegral} have a long history in 
mathematical physics; they appear in contexts ranging from lattice gauge
theory to quantum chromodynamics and string theory, see 
e.g.~\cite{BeenakerBrouwer96,DeWittHooft,GrossTaylor,Samuel,Xu97}.  
In the context of free probability and random matrices, these
integrals were treated by Collins \cite{Collins2003} and Collins--\'Sniady \cite{CollinsSniady2004}, 
who proved that
	\begin{equation}
		\label{eqn:WeingartenConvolutionFormula}
		I_N(i,j,i',j') = \sum_{({\pi_1},{\pi_2}) \in \Sy{d}^2} \delta_{i',i{\pi_1}} \delta_{j',j{\pi_2}}
		\Wg_N({\pi_1},{\pi_2}),
	\end{equation}
where 
	\begin{equation*}
		\label{eqn:WeingartenFunction}
		\Wg_N \colon  \Sy{d}^2 \to \mathbb{Q}
	\end{equation*}
is a special function on pairs of permutations  which they named the 
\emph{Weingarten function}.  

There are now several descriptions of the Weingarten function available; in this paper, 
we will use a series expansion of $\Wg_N$ obtained by Novak \cite{Novak2010}
and Matsumoto--Novak \cite{MatsumotoNovak2013}, 
which is explained in \cref{sec:MeanValueAsymptotics}.  
For now, plugging \eqref{eqn:WeingartenConvolutionFormula} into our calculation above 
eliminates the indices $i',j'$ and produces the formula
	\begin{equation}
	\label{eq:here-we-use-classical-independence}
		\tau_N=\frac{1}{N} \sum_{({\pi_1},{\pi_2}) \in \Sy{d}^2} \Wg_N({\pi_1},{\pi_2})
 		\ \E[S^A_{{\pi_1}}]\ \E[S^B_{{\pi_2}^{-1}\gamma}],
	\end{equation}
where
	\begin{align}
	\label{eqn:here-we-want-to-shuffle-factors1}
		S^A_{{\pi_1}}
		:=& \sum_{i\colon  [d] \to [N]} \prod_{k=1}^d \left(A_N^{p(k)}\right)_{i(k)i{\pi_1}(k)} \\
	\intertext{and} 
	\label{eqn:here-we-want-to-shuffle-factors2}
		S^B_{{\pi_2}^{-1}\gamma} :=&
		\sum_{i\colon  [d] \to [N]} \prod_{k=1}^d \left(B_N^{q(k)}\right)_{i(k)i{\pi_2}^{-1}\gamma(k)}.
	\end{align}

Our goal now is to express the operators \eqref{eqn:here-we-want-to-shuffle-factors1} and
\eqref{eqn:here-we-want-to-shuffle-factors2} in terms of the operators 
	\begin{equation*}
		\tr(A_N),\tr(A_N^2),\dots \quad\text{ and }\quad
		\tr(B_N),\tr(B_N^2),\dots,
	\end{equation*}
a task which is non-trivial due to the fact that the matrix elements of $A_N$ and
$B_N$ do not commute.
It is advantageous to lift this problem to the universal enveloping
algebra $\U(\mathfrak{gl}_N)$.

\subsection{Casimirs and Biasimirs}
Let $Z_N$ be the $N \times N$ matrix over $\U(\mathfrak{gl}_N)$ with elements 
\[(Z_N)_{ij}=\hbar\ e_{ij}.\]  
This matrix was introduced by Perelomov and Popov
\cite{PerelomovPopov1967}, who studied traces of its powers,
	\begin{equation*}
		C_k := \Tr Z_N^k = \sum_{i\colon [k] \rightarrow [N]}
		(Z_N)_{i(1)i(2)} \cdots (Z_N)_{i(k)i(1)}, \quad k \in \N^*,
	\end{equation*}
which they called \emph{higher Casimirs}, see also \cite{Zelobenko73}.  
This nomenclature stems from the fact 
that, up to a multiplicative factor of $\hbar^2$, the element $C_2$ coincides with 
the usual Casimir element which resides in the center $\mathcal{Z}_N$ of 
$\U(\mathfrak{gl}_N)$.  The following Theorem summarizes the main properties of 
higher Casimirs.

	\begin{theorem}[\cite{PerelomovPopov1967,Zelobenko73}]
	\label{thm:PerelomovPopov}
		The higher Casimirs generate $\mathcal{Z}_N$ as a polynomial ring,
			\begin{equation*}
			\mathcal{Z}_N = \C[C_1,C_2,C_3,\dots].
			\end{equation*}
		Moreover, if $(\mathbf{X},\rho)$ is the irreducible representation 
		of $\group{GL}_N(\C)$ indexed by the particle configuration
		$c_1 > \dots > c_N$ on $\lattice$, the image of $C_k$ in this representation
		is the scalar operator
			\begin{equation*}
				\rho(C_k) = \wp_k(c_1,\dots,c_N)I_{\mathbf{X}}
			\end{equation*}
		with eigenvalue
			\begin{equation*}
				\wp_k(c_1,\dots,c_N) 
				=\sum_{i=1}^N \prod_{j \neq i}  \left(1-\frac{\hbar}{c_i-c_j} \right)c_i^k.
			\end{equation*}
	\end{theorem}

Note that traces of powers of our quantum random matrices $A_N$ and $B_N$, 
which are operators acting in $\V_N \otimes \W_N$, are essentially images of higher 
Casimirs in irreducible representations;
more precisely, we have
	\begin{equation*}
	\label{eqn:TracesAreImagesOfCasimirs}
		\begin{split}
			\Tr(A_N^k)
				&= \rho_N(C_k) \otimes I_{\W_N}, \\
			\Tr(B_N^l)
				&= I_{\V_N} \otimes \sigma_N(C_l).
		\end{split}
	\end{equation*}
In particular, by \cref{thm:PerelomovPopov}, these traces are the following
scalar operators,
	\begin{equation*}
	\label{eqn:TracesAreScalarOperators}
		\begin{split}
			\Tr(A_N^k)
				&= \wp_k(a_1^{(N)},\dots,a_N^{(N)})I_{\V_N} \otimes I_{\W_N}, \\
			\Tr(B_N^l) 
				&= \wp_k(b_1^{(N)},\dots,b_N^{(N)})I_{\V_N} \otimes I_{\W_N}.
		\end{split}
	\end{equation*}
We conclude that, for any $k,l \in \N^*$, the operators $\Tr(A_N^k)$ and $\Tr(A_N^l)$
are classically independent quantum random variables in $(\A_N,\E)$ with known 
distributions, and similarly for the operators $\Tr(B_N^k),\Tr(B_N^l)$.

In order to understand the operators $S^A_{\pi_1},S^B_{{\pi_2}^{-1}\gamma}$ which
appear in our formula \eqref{eq:here-we-use-classical-independence} for $\tau_N$,
we must understand certain elements of $\U(\mathfrak{gl}_N)$ which further 
generalize higher Casimirs.  More precisely, we have that
	\begin{equation}
	\label{eqn:SfunctionsAreImagesOfBiasimirs}
		\begin{split}
			S^A_{{\pi_1}}
				&= \rho_N(C_{{\pi_1}}^{(p)}) \otimes I_{\W_N}, \\
			S^B_{{\pi_2}^{-1}\gamma}
				&= I_{\V_N} \otimes \sigma_N(C_{{\pi_2}^{-1}\gamma}^{(q)}),
		\end{split}
	\end{equation}
where, for any permutation $\pi \in \Sy{d}$ and function $r \colon [d] \to [N]$, we define
	\begin{equation}
		\label{eqn:Biasimir}
		C_\pi^{(r)} := \sum_{i\colon[d] \rightarrow [N]}
		(Z_N^{r(1)})_{i(1)i\pi(1)} \cdots (Z_N^{r(d)})_{i(d)i\pi(d)}.
	\end{equation}
Elements in $\U(\mathfrak{gl}_N)$ of the form \eqref{eqn:Biasimir} were first considered by 
Biane in \cite{Biane98}, and we shall refer to them as \emph{Biasimirs}, a portmanteau
of ``Biane'' and ``Casimir''. 
 Indeed, if $\pi = \gamma$ is the full forward cycle
in $\Sy{d}$, then $C_{\pi}^{(r)}$ reduces to the higher Casimir $C_{|r|}$.  	

Let us look at some examples of Biasimirs.  As an easy example, take $d=5$ 
and $\pi \in \Sy{5}$ to be the permutation $\pi = (1\ 2\ 3)(4\ 5)$.  Then, for any 
$r \colon [5] \rightarrow [N]$, we have
	\begin{align*}
		C_\pi^{(r)} &= \sum_{i\colon [5] \rightarrow [N]} 
			(Z_N^{r(1)})_{i(1)i(2)}(Z_N^{r(2)})_{i(2)i(3)}(Z_N^{r(3)})_{i(3)i(1)}
				(Z_N^{r(4)})_{i(4)i(5)}(Z_N^{r(5)})_{i(5)i(4)} \\
		&=C_{r(1)+r(2)+r(3)}C_{r(4)+r(5)}.
	\end{align*}
More generally, whenever $\pi \in \Sy{d}$ is a \emph{canonical permutation}, 
i.e.~a permutation of the form
	\begin{equation}
         \label{eq:canonical}
		\pi = (1\ 2\ \dots\ n_1)(n_1+1\ n_1+2\ \dots\ n_1+n_2) \cdots,
	\end{equation}
for some composition $(n_1,n_2,\dots)$ of $d$, the corresponding Biasimir will
be a simple monomial function of Casimirs.  To be precise, if $\pi = \gamma_1 \gamma_2 
\cdots \gamma_k$ is the disjoint cycle decomposition of a canonical 
permutation $\pi$, then
	\begin{equation}
	\label{eqn:ClassicalComponent}
		C_\pi^{(r)} = \prod_{j=1}^k C_{\sum_{i \in \gamma_j} r(i)}.
	\end{equation}	

Biasimirs corresponding to non-canonical permutations are more complicated 
functions of higher Casimirs.  

\begin{example} 
\label{example:Biasimir-bad}
Consider the Biasimir of 
degree $d=3$ corresponding to the non-canonical permutation $\pi = (1\ 3\ 2)$
and some general power function $r$ such that $r(2)=r(3)=1$,
	\begin{equation*}
		C_\pi^{(r)} = 
		\sum_{i\colon[3] \rightarrow [N]} (Z_N^{r(1)})_{i(1)i(3)} (Z_N)_{i(2)i(1)} 
		(Z_N)_{i(3)i(2)}.
	\end{equation*}
This is \emph{not} the higher Casimir $C_{r(1)+2}$, because the factors in each
term of the sum are in the wrong order.  However,
we can sort the letters in each summand using the commutation relations

Carrying this out and summing over all $i\colon[3] \rightarrow [N]$, we obtain
	\begin{align*}
		C_\pi^{(r)} 
		&= \Tr Z_N^{r(1)} Z_N Z_N + \hbar \Tr Z_N^{r(1)} \Tr Z_N 
		-\hbar N \Tr Z_N^{r(1)} Z_N \\
\nonumber
		&= C_{r(1)+2} + \hbar C_{r(1)}\ C_{1} - \hbar N\ C_{r(1)+1}.
	\end{align*}
\end{example}

In general, we have the following polynomial representation of Biasimirs in terms of 
higher Casimirs. 
\begin{proposition}
\label{prop:BianePolynomials}
For any permutation $\pi \in \Sy{d}$,
and any function $r\colon [d] \rightarrow \N$,
there exist unique polynomials $\mathbf{P}_\pi^{(r)}$ and $\mathbf{Q}_\pi^{(r)}$ in 
$|r|$ and $|r|+2$ variables, respectively, such that 
	\begin{equation*}
	C_\pi^{(r)} = \mathbf{P}_\pi^{(r)}(C_1,\dots,C_{|r|}) + 
	\hbar \mathbf{Q}_\pi^{(r)}(\hbar,N,C_1,\dots,C_{|r|})
	\end{equation*}
holds for all $N \in \N^*$.
\end{proposition}
\begin{proof}
This result is a generalization of a result of Biane \cite[Lemma 8.4.1]{Biane98}
who considered the special case when $r=\mathbf{1}_d \colon [d]\rightarrow \N$
is the constant function, equal to $1$. 
It might seem a bit worrying that Biane's proof uses a result which is not
quite correct, namely \cite[Lemma 8.3]{Biane98},
nevertheless we shall provide a corrected version of the latter result in 
\cref{lem:bad-things-that-can-happen-while-conjugating}.

The general case follows by an observation that there exists a natural permutation
$\pi'\in\Sy{|r|}$ with the property that 
\begin{equation}
    \label{eq:yes-you-can-make-it-flat}
    C_\pi^{(r)}=C_{\pi'}^{(\mathbf{1}_{|r|})}   
\end{equation}
is equal to the corresponding Biasimir with all exponents equal to $1$.
This permutation $\pi'$ is obtained from $\pi\colon [d]\to [d]$ by replacing each element 
$i\in[d]$ by its $r(i)$ copies which will be denoted by 
$i+\varepsilon, i+2 \varepsilon, \dots, i+ r(i) \varepsilon$, where $\varepsilon>0$ is an
infinitesimally small positive number.
Permutation $\pi'$ maps the rightmost copy of $i$ to the
leftmost copy of $\pi(i)$:
\[ \pi' \colon i+r(i) \varepsilon \mapsto \pi(i)+ \varepsilon \]
 and it maps each non-rightmost copy to its neighbor on the right:
\[ \pi' \colon i+k \varepsilon \mapsto i + (k+1) \varepsilon \qquad \text{for } 1\leq k<r(i).\]
The permutation $\pi'$ acts on the ordered set 
\[ \{ i + k \varepsilon : i\in [d], k\in [r(i)] \}; \]
by changing the labels in a way which preserves the order, $\pi'$ can be viewed as 
a usual permutation in $\Sy{|r|}$. 
\end{proof}

We refer to the polynomials $\mathbf{P}_\pi^{(r)}$ and $\mathbf{Q}_\pi^{(r)}$ as the
``classical'' and ``quantum'' components of the Biasimir $C_\pi^{(r)}$.  The classical
component is simple, being given by the right hand side of
formula \eqref{eqn:ClassicalComponent} above; the quantum component is more complicated.
Returning to \cref{example:Biasimir-bad} 
where $\pi \in \Sy{3}$ is the cyclic permutation $\pi=(1\ 3\ 2)$
and $r(2)=r(3)=1$ we have
	\begin{align*}
		\mathbf{P}_\pi^{(r)}(C_1,\dots,C_{|r|}) &= C_{r(1)+2}, \\
		\mathbf{Q}_\pi^{(r)}(\hbar,N,C_1,\dots,C_{|r|}) &= C_{r(1)} C_{1} -
		N\ C_{r(1)+1}
	\end{align*}
for the classical and quantum components of the Biasimir $C_\pi^{(r)}$.

Let us view the quantum component $\mathbf{Q}_\pi^{(r)}$ of a given 
Biasimir $C_\pi^{(r)}$ as an element of the polynomial ring $\Z[\hbar][N,C_1,\dots,C_{|r|}]$.
On this polynomial ring we impose the grading in which each variable
$N,C_1,\dots,C_{|r|}$ has degree one.
Let $\cyc(\pi)$ denote the number of factors in the decomposition of 
$\pi$ into disjoint cyclic permutations, and let
$\aex(\pi)$ denote the number of \emph{antiexceedances} of the permutation $\pi\in\Sy{d}$,
that is the number of indices $i \in [d]$ such that $\pi(i) \leq i$.  
The following Proposition
is a corrected version of \cite[Proposition 8.5]{Biane98}, see \cref{sec:erratum} for the proof
and further discussion.

	\begin{proposition}
	\label{prop:Antiexceedance}
		The degree of the classical component $\mathbf{P}_\pi^{(r)}$ is $\cyc(\pi)$.
		The degree of the quantum component $\mathbf{Q}_\pi^{(r)}$ is at most $\aex(\pi)$.
	\end{proposition}

\subsection{Classical/Quantum decomposition}
\newcommand{\PPA}[1]{\overline{\mathbf{P}}_{#1}^{\mathbf{A}}}
\newcommand{\QQA}[1]{\overline{\mathbf{Q}}_{#1}^{\mathbf{A}}}
\newcommand{\PPB}[1]{\overline{\mathbf{P}}_{#1}^{\mathbf{B}}}
\newcommand{\QQB}[1]{\overline{\mathbf{Q}}_{#1}^{\mathbf{B}}}

We are now ready to obtain the decomposition \eqref{eqn:ClassicalQuantumDecomposition}
of $\tau_N$ into classical and quantum parts.  Let us return to the formula 
\eqref{eq:here-we-use-classical-independence} for $\tau_N$, and consider a particular
term in the sum corresponding to the pair $({\pi_1},{\pi_2}) \in \Sy{d}^2$.  

First, by \cref{prop:BianePolynomials}, we have that
	\begin{equation*}
		S^A_{\pi_1} = \mathbf{P}_{\pi_1}^{(p)}(\Tr A_N,\dots,\Tr A_N^{|p|}) +
		\hbar \mathbf{Q}_{\pi_1}^{(p)}(\hbar,N,\Tr A_N,\dots,\Tr A_N^{|p|}).
	\end{equation*}
Let us rewrite this in terms of normalized traces.  Put
	\begin{align*}
		\PPA{{\pi_1}} &:=\frac{1}{N^{\cyc({\pi_1})}} 
			\mathbf{P}_{\pi_1}^{(p)}\left( \Tr A_N,\dots,\Tr A_N^{|p|}\right), \\
		\QQA{{\pi_1}} &:=\frac{1}{N^{\aex({\pi_1})}} 
                             \mathbf{Q}_{\pi_1}^{(p)}\left(\hbar,N,\Tr A_N,\dots,\Tr A_N^{|p|}\right).
           \end{align*}
From \eqref{eqn:ClassicalComponent}, $\PPA{{\pi_1}}$ is an
explicit polynomial in the operators
\[ \tr A_N, \tr A_N^2,\dots, \tr A_N^{|p|}, \]
while \cref{prop:Antiexceedance} implies that $\QQA{{\pi_1}}$ is a polynomial in the 
numbers $\hbar,N^{-1}$ and the operators
\begin{equation*}
\label{eq:useful-quantites-A}
 \tr A_N, \tr A_N^2,\dots, \tr A_N^{|p|}.
\end{equation*}
We thus have
	\begin{equation*}
		S^A_{\pi_1} = N^{\cyc({\pi_1})}\PPA{{\pi_1}} + \hbar N^{\aex({\pi_1})}\QQA{{\pi_1}}.
	\end{equation*}
Now we apply the expectation $\E$ to both sides of this identity in $\A_N$ to 
get an identity in $\C$.  Because traces of powers of $A_N$ are 
classically independent,
we have proved the following result.
\begin{proposition}
\label{prop:PPA-QQA-asymptoticalyO1}
Define
	\begin{align*} 
		\PPA{{\pi_1}}(N)&:=\E\PPA{{\pi_1}} =\frac{1}{N^{\cyc({\pi_1})}} 
			\mathbf{P}_{\pi_1}^{(p)}\left( \E\Tr A_N,\dots,\E\Tr A_N^{|p|}\right), \\
		\QQA{{\pi_1}}(N)&:=\E\QQA{{\pi_1}} =\frac{1}{N^{\aex({\pi_1})}} 
                             \mathbf{Q}_{\pi_1}^{(p)}\left(\hbar,N,\E\Tr A_N,\dots,\E\Tr A_N^{|p|}\right).
           \end{align*}
Then $\PPA{{\pi_1}}(N)$ is a polynomial in the numbers
\[ \E\tr A_N, \E\tr A_N^2,\dots, \E\tr A_N^{|p|}, \]
and $\QQA{{\pi_1}}(N)$ is a polynomial in the numbers 
\begin{equation*}
 \hbar,N^{-1},\E\tr A_N, \E\tr A_N^2,\dots, \E\tr A_N^{|p|}.
\end{equation*}
\end{proposition}

\noindent
We conclude that
	\begin{equation*}
		\E[S^A_{\pi_1}] = N^{\cyc({\pi_1})}\PPA{{\pi_1}}(N) + \hbar N^{\aex({\pi_1})}\QQA{{\pi_1}}(N).
	\end{equation*}

We now record the counterpart of the above for the matrix $B$;
the calculations are fully analogous to those just performed for $A$. 
We have that
	\begin{equation*}
		S^B_{{\pi_2}^{-1}\gamma} = \mathbf{P}_{{\pi_2}^{-1}\gamma}^{(q)}(\Tr B_N,\dots,\Tr B_N^{|q|}) +
		\hbar\mathbf{Q}^{(q)}_{{\pi_2}^{-1}\gamma}(\hbar,N,\Tr B_N,\dots,\Tr B_N^{|q|}).
	\end{equation*}
Once again, let us rewrite this in terms of normalized traces.  Put
	\begin{align*}
		\PPB{{\pi_2}^{-1}\gamma} &:=\frac{1}{N^{\cyc({\pi_2}^{-1}\gamma)}} 
			\mathbf{P}_{{\pi_2}^{-1}\gamma}^{(q)}\left( \Tr B_N,\dots,\Tr B_N^{|q|}\right), \\
		\QQB{{\pi_2}^{-1}\gamma} &:=\frac{1}{N^{\aex({\pi_2}^{-1}\gamma)}} 
                             \mathbf{Q}_{{\pi_2}^{-1}\gamma}^{(q)}\left(\hbar,N,\Tr B_N,\dots,\Tr B_N^{|q|}\right).
           \end{align*}
From \eqref{eqn:ClassicalComponent}, $\PPB{{\pi_2}^{-1}\gamma}$ 
is an explicit polynomial in the operators
\[ \tr B_N, \tr B_N^2,\dots, \tr B_N^{|q|}, \]
while \cref{prop:Antiexceedance} implies that $\QQB{{\pi_2}^{-1}\gamma}$ is a polynomial in the 
numbers $\hbar,N^{-1}$ and the operators
\begin{equation*}
\label{eq:useful-quantites-B1}
 \tr B_N, \tr B_N^2,\dots, \tr B_N^{|q|}.
\end{equation*}
We thus have
	\begin{equation*}
		S^B_{{\pi_2}^{-1}\gamma} = N^{\cyc({\pi_2}^{-1}\gamma)}\PPB{{\pi_2}^{-1}\gamma} + 
		\hbar N^{\aex({\pi_2}^{-1}\gamma)}\QQB{{\pi_2}^{-1}\gamma}.
	\end{equation*}
Once again, we apply $\E$ to both sides of this identity in $\A_N$ to 
get an identity in $\C$.  As above, we declare
	\begin{align*}
		\PPB{{\pi_2}^{-1}\gamma}(N)&:=\E\PPB{{\pi_2}^{-1}\gamma} =
				\frac{1}{N^{\cyc({\pi_2}^{-1}\gamma)}} 
			\mathbf{P}_{{\pi_2}^{-1}\gamma}^{(q)}\left( \E\Tr B_N,\dots,\E\Tr B_N^{|q|}\right), \\
		\QQB{{\pi_2}^{-1}\gamma}(N)&:=\E\QQB{{\pi_2}^{-1}\gamma} =
			\frac{1}{N^{\aex({\pi_2}^{-1}\gamma)}} 
                             \mathbf{Q}_{{\pi_2}^{-1}\gamma}^{(q)}
                             \left(\hbar,N,\E\Tr B_N,\dots,\E\Tr B_N^{|q|}\right).
           \end{align*}
The first of these is a polynomial in the numbers
\[ \E\tr B_N, \E\tr B_N^2,\dots, \E\tr B_N^{|q|}, \]
while the second is a polynomial in the numbers 
\begin{equation*}
 \hbar,N^{-1},\E\tr B_N, \E\tr B_N^2,\dots, \E\tr B_N^{|q|}.
\end{equation*}
We conclude that
	\begin{equation*}
		\E[S^B_{{\pi_2}^{-1}\gamma}] = N^{\cyc({\pi_2}^{-1}\gamma)}\PPB{{\pi_2}^{-1}\gamma}(N) 
		+ \hbar N^{\aex({\pi_2}^{-1}\gamma)}\QQB{{\pi_2}^{-1}\gamma}(N).
	\end{equation*}

\bigskip

Putting these two calculations together, we compute the $({\pi_1},{\pi_2})$ term of 
$\tau_N$ as
	\begin{multline*}
		\Wg_N({\pi_1},{\pi_2})\ \E[S^A_{\pi_1}]\ \E[S^B_{{\pi_2}^{-1}\gamma}] 
		= \\ \Wg_N({\pi_1},{\pi_2})
			\left( N^{\cyc({\pi_1})}\PPA{{\pi_1}}(N) + 
			\hbar N^{\aex({\pi_1})}\QQA{{\pi_1}}(N) \right) \times \\
			\times \left( N^{\cyc({\pi_2}^{-1}\gamma)}\PPB{{\pi_2}^{-1}\gamma}(N) 
				+ \hbar N^{\aex({\pi_2}^{-1}\gamma)}\QQB{{\pi_2}^{-1}\gamma}(N) \right).
	\end{multline*}        
Expanding the brackets and summing $({\pi_1},{\pi_2})$ over $\Sy{d}^2$, we arrive at the 
classical/quantum decomposition of the mixed moment $\tau_N$.
	
	\begin{theorem}
	\label{thm:ClassicalQuantumDecomposition}
	We have
		\begin{equation*}
			\tau_N = \Classical_N + \hbar_N \Quantum_N,
		\end{equation*}
	where
		\begin{multline*}
			\Classical_N = \\ \frac{1}{N}\sum_{({\pi_1},{\pi_2}) \in \Sy{d}^2} 
			N^{\cyc({\pi_1})+\cyc({\pi_2}^{-1}\gamma)}\Wg_N({\pi_1},{\pi_2})\ 
			\PPA{{\pi_1}}(N)\ \PPB{{\pi_2}^{-1}\gamma}(N)
		\end{multline*}
	and
		\begin{multline*}
			\Quantum_N =\\ \begin{aligned} \hspace{5ex} & \frac{1}{N}\sum_{({\pi_1},{\pi_2}) \in \Sy{d}^2} 
			\bigg{(}
			N^{\cyc({\pi_1})+\aex({\pi_2}^{-1}\gamma)} \Wg_N({\pi_1},{\pi_2})\
				\PPA{{\pi_1}}(N)\ \QQB{{\pi_2}^{-1}\gamma}(N) \\
			&+
				N^{\aex({\pi_1})+\cyc({\pi_2}^{-1}\gamma)}\Wg_N({\pi_1},{\pi_2})\
				\QQA{{\pi_1}}(N)\ \PPB{{\pi_2}^{-1}\gamma}(N) \\
			&+  \hbar N^{\aex({\pi_1})+\aex({\pi_2}^{-1}\gamma)}\Wg_N({\pi_1},{\pi_2})\
				\QQA{{\pi_1}}(N)\ \QQB{{\pi_2}^{-1}\gamma}(N)
			\bigg{)}.
		\end{aligned} \end{multline*}
	\end{theorem}

\section{The (counter)example}
\label{sec:counterexample}

Before commencing a general analysis of the $N\to\infty$ limit of $\tau_N$,
let us examine a specific, concrete example.

Consider the specific choice of the polynomial in \eqref{eqn:QuantumMixedMoment}
given by
\begin{equation}
\label{eq:our-counterexample} 
\tau_N = \E\tr( A_N^r B_N^s A_N B_N A_N B_N)
\end{equation}
for some integers $r,s\geq 0$. 
This choice corresponds to $d=3$, $p=(r,1,1)$, $q=(s,1,1)$.
The Reader may fast-forward to \cref{sec:counterexample-conclusion}
where the quantum part of this particular $\tau_N$ is explicitly presented.

\subsection{Calculations}
Thanks to \eqref{eq:here-we-use-classical-independence} combined with
\eqref{eqn:SfunctionsAreImagesOfBiasimirs}, the quantity $\tau_N$
\[
\tau_N=\frac{1}{N} \sum_{{\pi_1},{\pi_2} \in \Sy{3}} \Wg_N({\pi_1},{\pi_2})
 		\ \tr_{\V_N} \rho_N(C_{{\pi_1}}^{(p)})\ \tr_{\W_N}  \sigma_N(C_{{\pi_2}^{-1}\gamma}^{(q)})   \]
can be 
expressed in terms of the Biasimirs $C^{(p)}_{\pi}$ and $C^{(q)}_{\pi}$
over the six permutations $\pi\in\Sy{3}$.
Four of these permutations are in the canonical form \eqref{eq:canonical} 
and the corresponding values of  $C^{(p)}_{\pi}$ and $C^{(q)}_{\pi}$
are given simply by \eqref{eqn:ClassicalComponent}.
The transposition $\pi=(1\ 3)$ is not canonical but a short thought
shows that also in this case a version of the formula \eqref{eqn:ClassicalComponent} applies.
The only more challenging case is $\pi=(1\ 3\ 2)$ which was already considered in
\cref{example:Biasimir-bad}. Thus we have 
\begin{equation}
\label{eq:casimir-concrete}
\left\{
\begin{aligned}
C^{(p)}_{\operatorname{id}} &= C_r C_1^2, & 
C^{(p)}_{(1\ 2)} &= C_{r+1} C_1, \\
C^{(p)}_{(1\ 3)} &= C_{r+1} C_1,          &
C^{(p)}_{(2\ 3)} &= C_r C_2,    \\
C^{(p)}_{(1\ 2\ 3)} &= C_{r+2},           &
C^{(p)}_{(1\ 3\ 2)} &= C_{r+2}+ \hbar_N C_r C_1 - \hbar_N N\  C_{r+1};
\end{aligned}
\right.
\end{equation}
analogous formulas give the values of $C_{\pi}^{(q)}$.

If we come back to the original formula \eqref{eq:here-we-use-classical-independence},
$\tau_N$ is expressed in terms of the quantities $S^A_\pi$ and $S^B_\pi$ 
which are even better suited for the purposes of asymptotic problems. In this context \eqref{eq:casimir-concrete}
becomes
\begin{equation}
\label{eq:casimir-concrete-2}
\left\{
\begin{aligned}
S^A_{\operatorname{id}} &= N^3 \tr A^r \left(\tr A \right)^2, & 
S^A_{(1\ 2)} &=            N^2 \tr A^{r+1} \tr A, \\
S^A_{(1\ 3)} &= N^2 \tr A^{r+1} \tr A^1,          &
S^A_{(2\ 3)} &= N^2 \tr A^r \tr A^2,    \\
S^A_{(1\ 2\ 3)} &=N  \tr A^{r+2},           \\
S^A_{(1\ 3\ 2)} &=N \tr A^{r+2}+ \hbar_N N^2 \tr A^r \tr A - \hbar_N N^2  \tr A^{r+1},
\hspace{-20ex}
\end{aligned}
\right.
\end{equation}
with $S^B_{\pi}$ given by analogous formulas.

The values of the Weingarten function are explicitly known
rational functions in $N$:
\begin{equation}
\label{eq:Weingarten-3}
 \Wg_N({\pi_1},{\pi_2}) = 
\begin{cases}
\frac{N^2-2}{N(N^2-1)(N^2-4)} & \text{if $\pi_1 \pi_2^{-1}=\operatorname{id}$}, \\
\frac{-1}{(N^2-1)(N^2-4)}     & \text{if $\pi_1 \pi_2^{-1}$ is a transposition}, \\
\frac{2}{N(N^2-1)(N^2-4)}     & \text{if $\pi_1 \pi_2^{-1}$ is a cycle of length $3$}.   
\end{cases}  
\end{equation}

An application of a computer algebra system to 
\eqref{eq:here-we-use-classical-independence} with the data
\eqref{eq:casimir-concrete-2} and \eqref{eq:Weingarten-3} 
gives an explicit but complicated formula for $\tau_N$
as a polynomial in the indeterminates 
\begin{equation}
\label{eq:variables}
\tr A, \tr A^2, \tr A^{r}, \tr A^{r+1}, \tr A^{r+2},
\tr B, \tr B^2, \tr B^{s}, \tr B^{s+1}, \tr B^{s+2}, \hbar_N
\end{equation}
and coefficients in the field $\mathbb{Q}(N)$ of rational functions in $N$.
The limit $\lim_{N\to\infty} \tau_N$ turns out to be a polynomial in the indeterminates
\eqref{eq:variables} with integer coefficients
which involves $14$ monomials.
Ten of these monomials do not involve the Planck constant $\hbar_N$;
it follows that
with respect to the decomposition \eqref{eqn:ClassicalQuantumDecomposition}
they correspond to the classical part $\Classical_N$.
The remaining four monomials which are divisible by $\hbar_N$
correspond to the quantum part $\Quantum_N$.
We shall review them in the following.

\subsection{The conclusion}
\label{sec:counterexample-conclusion}

By \cref{thm:Main} which we are about to prove 
(it follows also by direct inspection), the classical part 
$\Classical_N$ of $\tau_N$ with respect to the decomposition
\eqref{eqn:ClassicalQuantumDecomposition}
corresponds to the
terms given by free probability theory. 

Much more mysterious is the 
quantum part which in our case turns out to be given by
\[
\Quantum_N = \hbar_N \left(  \tr A^{r+1} - \tr A \tr A^r  \right) 
\left(  \tr B^{s+1} - \tr B \tr B^s  \right).
\]
From our perspective it is important that 
this quantum part is clearly non-zero as soon as $r,s\geq 1$.
This shows that the assumption that $\hbar_N\to 0$
is indeed necessary in \cref{thm:Main} in order to have asymptotic freeness.

\medskip

Finally, we would like to point out that $\tau_N$ 
given by \eqref{eq:our-counterexample} is the simplest example
for which $\lim_{N\to\infty} \Quantum_N \neq 0$.
In fact, for any alternating product of four factors 
\[ \tau_N = \E\tr( A_N^r B_N^s A_N^t B_N^u) \]
with integers $r,s,t,u\geq 0$,
the corresponding quantum part is identically zero.
An interesting question, at present unresolved, is the following.

\begin{problem}
How big can the quantum part be?
We have seen in the above example that
$\Quantum_N=\Theta\left( \hbar_N \right)$ can occur;
\cref{cor:SemiclassicalAsymptotics} gives the much weaker bound
$\Quantum_N=O(1)$.
Are there examples for which this bound is saturated and 
$\Quantum_N=\Theta(1)$?
\end{problem}

\section{Mean Value Asymptotics}
\label{sec:MeanValueAsymptotics}
In this Section, we apply the exact results obtained in \cref{sec:MeanValuesPlanckScale}
to analyze the asymptotic behavior of the mixed moment $\tau_N$ 
in the limit where $N \to \infty$ and $\hbar_N \to 0$.  
We adopt the hypotheses of \cref{thm:Main}, which is to say that 
we henceforth assume the limits
	\begin{equation*}
		s_k:=\lim_{N \to \infty}\E\tr(A_N^k) \quad\text{ and }\quad t_k:=
		\lim_{N \to \infty}\E\tr(B_N^k)
	\end{equation*}
exist for each fixed $k \in \N^*$.

\subsection{The Weingarten function}
A key component of our asymptotic analysis will be an absolutely convergent series 
expansion for the Weingarten function which renders its asymptotic behavior transparent.

In order to state this expansion,
let us identify the symmetric group $\Sy{d}$ with its right Cayley graph, as generated by the conjugacy 
class of transpositions.  We denote by $|\cdot|$ the corresponding word norm, so 
that $|{\pi_1}^{-1}{\pi_2}|$ is the graph theory distance from ${\pi_1}$ to ${\pi_2}$, i.e.~the 
length of a geodesic path in the Cayley graph joining these two permutations.  
Equip the Cayley graph with the \emph{Biane--Stanley edge
labeling}, in which each edge corresponding to the transposition $(s\ t)$ is marked by
$t$, the larger of the two elements interchanged.  This edge labeling
was introduced in the context of enumerative combinatorics 
by Stanley \cite{Stanley} and Biane \cite{Biane2002} as a tool to 
relate parking functions and noncrossing partitions.  \cref{fig:Cayley} shows 
$\Sy{4}$ with the Biane-Stanley labeling, where $2$-edges are drawn in blue,
$3$-edges in yellow, and $4$-edges in red.  

A walk on $\Sy{d}$ is said to be 
\emph{monotone} if the labels of the edges it traverses form a weakly increasing
sequence.  The fundamental fact we need \cite{Novak2010,MatsumotoNovak2013} 
is that the Weingarten function expands as a generating function for monotone walks: we have
	\begin{equation}
		\label{eqn:WeingartenRaw}
		\Wg_N({\pi_1},{\pi_2}) = \frac{1}{N^d} \sum_{r=0}^\infty (-1)^r
		\frac{\vec{W}^r({\pi_1},{\pi_2})}{N^r},
	\end{equation}
where $\vec{W}^r({\pi_1},{\pi_2})$ is the number of $r$-step monotone walks on
$\Sy{d}$ which begin at the permutation ${\pi_1}$ and end at the permutation ${\pi_2}$.
This series is absolutely convergent provided $N \geq d$, but divergent if 
$N < d$ (this divergence is a related to the \emph{De Wit--'t Hooft 
anomalies} in $\group{U}(N)$ lattice gauge theory, see e.g.~\cite{DeWittHooft,Morozov1,Samuel}). 

\begin{figure}
	\includegraphics{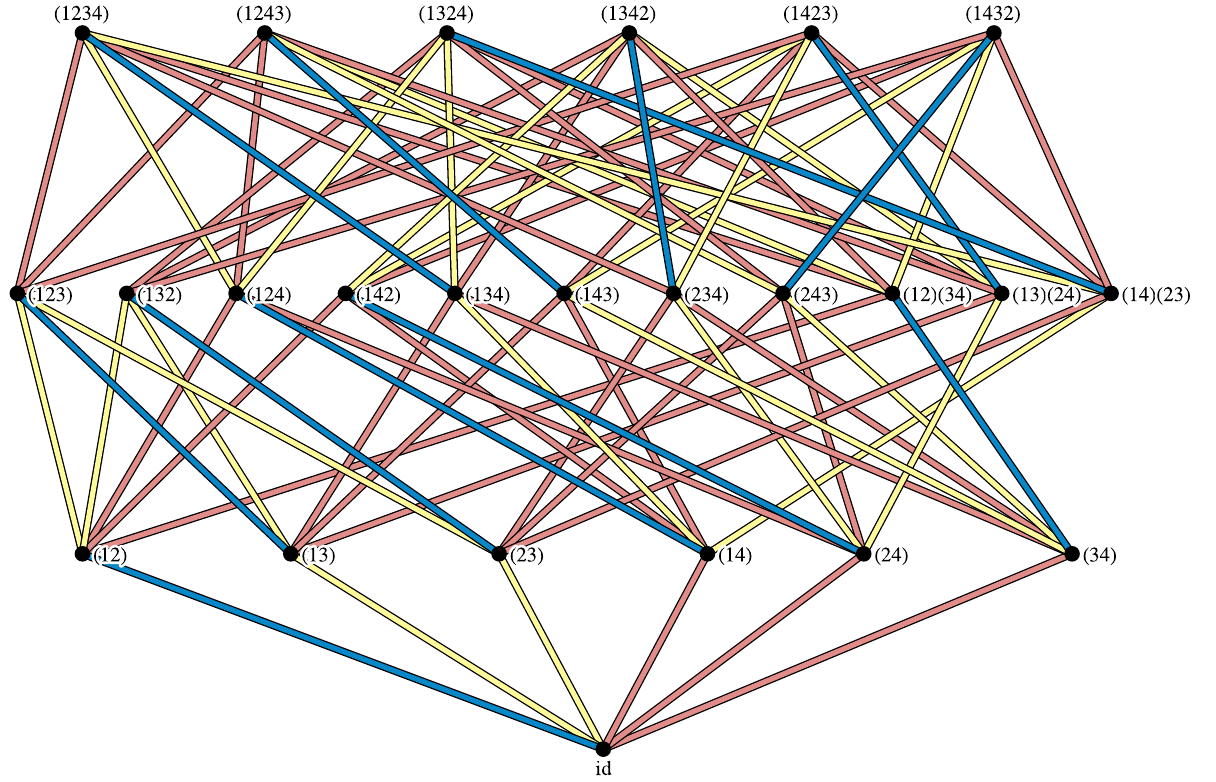}
	\caption{\label{fig:Cayley} $\Sy{4}$ with the Biane--Stanley
	edge-labeling (figure by M.~LaCroix).}
\end{figure}

Since $\vec{W}^r({\pi_1},{\pi_2}) =  \vec{W}^r(\id,{\pi_1}^{-1}{\pi_2})$, and since
every permutation is either even or odd, the number $\vec{W}^r({\pi_1},{\pi_2})$ is nonzero if and only if 
$r=|{\pi_1}^{-1}{\pi_2}|+2g$ for some $g \in \N$.  We may thus rewrite \eqref{eqn:WeingartenRaw}
as 
	\begin{equation}
		\label{eqn:WeingartenReduced}
		\Wg_N({\pi_1},{\pi_2}) = \frac{(-1)^{|{\pi_1}^{-1}{\pi_2}|}}{N^{d+|{\pi_1}^{-1}{\pi_2}|}}
		\sum_{g=0}^\infty \frac{\vec{W}_g({\pi_1},{\pi_2})}{N^{2g}},
	\end{equation}
where $\vec{W}_g({\pi_1},{\pi_2}) := \vec{W}^{|{\pi_1}^{-1}{\pi_2}|+2g}({\pi_1},{\pi_2})$.
The formulas \eqref{eqn:WeingartenConvolutionFormula} and 
\eqref{eqn:WeingartenReduced} may be effectively combined to yield a sort of Feynman calculus
for unitary matrix integrals, in which the role of Feynman diagrams is played
by monotone walks on symmetric groups, see e.g.~\cite{GGN4,GGN5}.

\subsection{Quantum asymptotics}
	We now show that the quantum part of $\tau_N$ can be controlled
	even for $\hbar_N=\hbar$ fixed.  In order to do this, we introduce a new
	permutation statistic defined by 
		\begin{equation*}
		\label{eqn:Defect}
			\defect(\pi) := \aex(\pi) - \cyc(\pi), \quad \pi \in \Sy{d}.
		\end{equation*}
	Moreover, for any $k\in \N^*$ and ${\pi_1},\dots,\pi_k \in \Sy{d}$, the quantity
	\begin{equation*}
		\label{eqn:Genus}
	\genus({\pi_1},\dots,\pi_k):=\frac{|{\pi_1}| + \dots + |\pi_k| - |{\pi_1}\dots\pi_k|}{2}
	\end{equation*}
is a nonnegative integer; we
refer to it as the \emph{genus} of the $k$-tuple $({\pi_1},\dots,\pi_k)$.
		
The following combinatorial result is a minor extension of the 
original results of Biane \cite[page 173]{Biane98}, see \cref{sec:proof-of-lem:Defect} for a
detailed proof and discussion of the relation to Biane's work.

	\begin{lemma}
	\label{lem:Defect}
		For any $\pi \in \Sy{d}$, we have
			\begin{equation*}
				\defect(\pi) \geq 0.
			\end{equation*}

		Moreover, for any permutations ${\pi_1},{\pi_2} \in \Sy{d}$, we have
			\begin{equation}
				\label{eq:defects-nice}
				\defect({\pi_1}) + \defect({\pi_2}^{-1}\gamma) \leq 2 
				\genus({\pi_1},{\pi_1}^{-1}{\pi_2},{\pi_2}^{-1}\gamma).
			\end{equation}
	\end{lemma}

	\begin{theorem}
	\label{thm:QuantumAsymptotics}
		If $\hbar_N=O(1)$ as $N \to \infty$, then
		$\Quantum_N=O(1)$ as $N \rightarrow \infty$.
	\end{theorem}
	
	\begin{proof}
		By \cref{thm:ClassicalQuantumDecomposition},
		the quantum part of $\tau_N$ may be 
		written as 
			\begin{equation}
			    \label{eq:quantum-sum}
			\Quantum_N = \sum_{({\pi_1},{\pi_2}) \in \Sy{d}^2} 
			N^{\cyc({\pi_1})+\cyc({\pi_2}^{-1}\gamma)-1} \Wg_N({\pi_1},{\pi_2})
			\ \mathbf{R}_{({\pi_1},{\pi_2})}(N),
			\end{equation}
		where 
			\begin{align*}
			\mathbf{R}_{({\pi_1},{\pi_2})}(N) =&
			N^{\defect({\pi_2}^{-1}\gamma)}
				\ \PPA{{\pi_1}}(N)\ \QQB{{\pi_2}^{-1}\gamma}(N) \\
				&+N^{\defect({\pi_1})}
				\ \QQA{{\pi_1}}(N)\ \PPB{{\pi_2}^{-1}\gamma}(N)\\
			&+\hbar_N N^{\defect({\pi_1})+\defect({\pi_2}^{-1}\gamma)}
				\ \QQA{{\pi_1}}(N)\ \QQB{{\pi_2}^{-1}\gamma}(N).
			\end{align*}
		We will show that each term in the sum \eqref{eq:quantum-sum} is $O(1)$.

		By the first part of \cref{lem:Defect}, nonnegativity of the defect statistic,
                as well as by \cref{prop:PPA-QQA-asymptoticalyO1}
                which shows that
		$\PPA{{\pi_1}}(N)=O(1)$, $\QQA{{\pi_1}}(N)=O(1)$ and its counterpart
                for the matrix $B$
		we have
			\begin{equation*}
				\mathbf{R}_{({\pi_1},{\pi_2})}(N) = 
				O\left(N^{\defect({\pi_1})+\defect({\pi_2}^{-1}\gamma)}\right)
			\end{equation*}
		for each $({\pi_1},{\pi_2}) \in \Sy{d}^2$.  
		
		Now, let us
		consider the order of the factor
			\begin{equation*}
				N^{\cyc({\pi_1})+\cyc({\pi_2}^{-1}\gamma)-1} \Wg_N({\pi_1},{\pi_2}).
			\end{equation*}
		Invoking the expansion \eqref{eqn:WeingartenReduced}, for 
		any $N \geq d$ we have
			\begin{multline*}
				N^{\cyc({\pi_1})+\cyc({\pi_2}^{-1}\gamma)-1} \Wg_N({\pi_1},{\pi_2})
				 \\ 
                         \begin{aligned} &= N^{\cyc({\pi_1})+\cyc({\pi_2}^{-1}\gamma)-1}
				\frac{(-1)^{|{\pi_1}^{-1}{\pi_2}|}}{N^{d+|{\pi_1}^{-1}{\pi_2}|}}
				\sum_{g=0}^{\infty} \frac{\vec{W}_g({\pi_1},{\pi_2})}{N^{2g}} \\
				 &=N^{-|{\pi_1}|-|{\pi_1}^{-1}{\pi_2}| - |{\pi_2}^{-1}\gamma| + |\gamma|}
				(-1)^{|{\pi_1}^{-1}{\pi_2}}
				\sum_{g=0}^{\infty} \frac{\vec{W}_g({\pi_1},{\pi_2})}{N^{2g}} \\
				&= O(N^{-2\genus({\pi_1},{\pi_1}^{-1}{\pi_2},{\pi_2}^{-1}\gamma)}).
			\end{aligned} \end{multline*}
			
		We conclude that each term of $\Quantum_N$ is of order		
			\begin{multline*}
				N^{\cyc({\pi_1})+\cyc({\pi_2}^{-1}\gamma)-1} \Wg_N({\pi_1},{\pi_2})
			\mathbf{R}_{({\pi_1},{\pi_2})}(N)= \\
				O\left(N^{\defect({\pi_1})+\defect({\pi_2}^{-1}\gamma)-
				2 \genus({\pi_1},{\pi_1}^{-1}{\pi_2},{\pi_2}^{-1}\gamma)}\right),
			\end{multline*}
		and hence is $O(1)$ by the second part of \cref{lem:Defect}.
	\end{proof}

\subsection{Classical asymptotics}
We now deal with the asymptotics of the classical part of $\tau_N$.

	\begin{theorem}
		\label{thm:ClassicalAsymptotics}
		Under the hypotheses of \cref{thm:Main}, the
		classical part of $\tau_N$ admits, for each $N \geq d$,
		the absolutely convergent series expansion
			\begin{equation*}
				\Classical_N= \sum_{k=0}^{\infty} \frac{e_k(N)}{N^{2k}},
			\end{equation*}
		where 
			\begin{multline*}
				e_k(N) = \sum_{\substack{(g,h) \in \N^2 \\ g+h=k}} 
				\sum_{\substack{({\pi_1},{\pi_2}) \in \Sy{d}^2 \\
				\genus({\pi_1},{\pi_1}^{-1}{\pi_2},{\pi_2}^{-1}\gamma)=h}}
				(-1)^{|{\pi_1}^{-1}{\pi_2}|}\ \times \\ \times \vec{W}_g({\pi_1},{\pi_2})\
				\PPA{{\pi_1}}(N)\ \PPB{{\pi_2}^{-1}\gamma}(N).
			\end{multline*}
	\end{theorem}
	
	\begin{proof}
		According to \cref{thm:ClassicalQuantumDecomposition}
		and the expansion \eqref{eqn:WeingartenReduced},
		we have
			\begin{multline*}
				\Classical_N = 
                                \\ \begin{aligned} \hspace{5ex} &= \sum_{({\pi_1},{\pi_2}) \in \Sy{d}^2} 
				N^{\cyc({\pi_1})+\cyc({\pi_2}^{-1}\gamma)-1}\Wg_N({\pi_1},{\pi_2})
				\ \PPA{{\pi_1}}(N)\ \PPB{{\pi_2}^{-1}\gamma}(N) \\
				&= \sum_{({\pi_1},{\pi_2}) \in \Sy{d}^2} 
				N^{-2\genus({\pi_1},{\pi_1}^{-1}{\pi_2},{\pi_2}^{-1}\gamma)} 
                                (-1)^{|{\pi_1}^{-1}{\pi_2}|}
				\times \\ & \hspace{20ex} \times
				\PPA{{\pi_1}}(N)\ \PPB{{\pi_2}^{-1}\gamma}(N)
				\sum_{g=0}^{\infty} \frac{\vec{W}_g({\pi_1},{\pi_2})}{N^{2g}} \\
				&=\sum_{g,h=0}^{\infty} \frac{1}{N^{2(g+h)}}
				\sum_{\substack{({\pi_1},{\pi_2}) \in \Sy{d}^2 \\
				\genus({\pi_1},{\pi_1}^{-1}{\pi_2},{\pi_2}^{-1}\gamma)=h}} 
				(-1)^{|{\pi_1}^{-1}{\pi_2}|} \times \\ & \hspace{20ex} \times \vec{W}_g({\pi_1},{\pi_2})\
				\PPA{{\pi_1}}(N)\ \PPB{{\pi_2}^{-1}\gamma}(N) \\
				&= \sum_{k=0}^{\infty} \sum_{\substack{g,h\geq 0\\ g+h=k}} \frac{1}{N^k}
				\sum_{\substack{({\pi_1},{\pi_2}) \in \Sy{d}^2 \\
				\genus({\pi_1},{\pi_1}^{-1}{\pi_2},{\pi_2}^{-1}\gamma)=h}}
				(-1)^{|{\pi_1}^{-1}{\pi_2}|} \times \\ & \hspace{20ex} \times \vec{W}_g({\pi_1},{\pi_2})\
				\PPA{{\pi_1}}(N)\ \PPB{{\pi_2}^{-1}\gamma}(N).
				\end{aligned} \end{multline*}
	\end{proof}	

\subsection{Semiclassical asymptotics and freeness}
Combining \cref{thm:ClassicalAsymptotics,thm:QuantumAsymptotics},
we obtain the following corollary.

	\begin{corollary}
	\label{cor:SemiclassicalAsymptotics}
		For any sequence $\hbar_N=O(1)$, we have
			\begin{equation*}
				\tau_N = \Classical_N + O(\hbar_N)
			\end{equation*}	
		as $N \rightarrow \infty$.
		In particular, if $\hbar_N = o(N^{-2l})$ as $N \rightarrow \infty$, then
			\begin{equation*}
				\tau_N = \Classical_N + o(N^{-2l}) =\sum_{k=0}^l \frac{e_k(N)}{N^{2k}} + 
				o(N^{-2l}).
			\end{equation*}
	\end{corollary}

	Let us now explain how the
	the $l=0$ case of \cref{cor:SemiclassicalAsymptotics} yields
	the proof of \cref{thm:Main}.  Assuming only $\hbar_ N = o(1)$
	as $N \to \infty$, \cref{cor:SemiclassicalAsymptotics} implies
	\begin{equation*}
			\tau_N = \Classical_N + o(1) = e_0(N) + o(1)
		\end{equation*}
	as $N \rightarrow \infty$, with
		\begin{equation*}
			e_0(N) = \sum_{\substack{({\pi_1},{\pi_2}) \in \Sy{d}^2 \\ 
			\genus({\pi_1},{\pi_1}^{-1}{\pi_2},{\pi_2}^{-1}\gamma)=0}}
				(-1)^{|{\pi_1}^{-1}{\pi_2}|}\ \vec{W}_0({\pi_1},{\pi_2})\
				\PPA{{\pi_1}}(N)\ \PPB{{\pi_2}^{-1}\gamma}(N).
		\end{equation*}
	Now, under the hypotheses of \cref{thm:Main}, the limits
		\begin{equation*}
			\PPA{{\pi_1}}=\lim_{N \to \infty} \PPA{{\pi_1}}(N) 
			 \quad\text{ and }\quad
			\PPB{{\pi_2}^{-1}\gamma}=
			\lim_{N \to \infty}\PPB{{\pi_2}^{-1}\gamma}(N)
		\end{equation*}
	exists, and are polynomials in $x_1,\dots,x_{|p|}$ and 
	$y_1,\dots,y_{|q|}$ given explicitly by the universal form
	\eqref{eqn:ClassicalComponent} of the classical component.
	We thus have
		\begin{equation*}
			\lim_{N \to \infty} \tau_N 
			=  \sum_{\substack{({\pi_1},{\pi_2}) \in \Sy{d}^2 \\ 
			\genus({\pi_1},{\pi_1}^{-1}{\pi_2},{\pi_2}^{-1}\gamma)=0}}
				(-1)^{|{\pi_1}^{-1}{\pi_2}|}\ \vec{W}_0({\pi_1},{\pi_2})\
				\PPA{{\pi_1}}\ \PPB{{\pi_2}^{-1}\gamma}.
		\end{equation*}
	This is exactly asymptotic freeness, see Appendix \ref{app:FreeProb}.
	Hence, Theorem \ref{thm:Main} is proved.

\section{Covariance of BPP Observables}
\label{sec:Covariance}
In this section we explain how the mean value analysis carried out in 
Sections \ref{sec:MeanValuesPlanckScale} and 
\ref{sec:MeanValueAsymptotics}, which yields the semiclassical/large-dimension 
asymptotics of the $1$-point functions of BPP observables of the LR process, 
can be extended to higher correlation functions.  We limit our discussion
to the connected $2$-point functions (covariances)
	\begin{equation*}
		\langle \p_{k_1}^{(N)}\p_{k_2}^{(N)} \rangle_c = 
		\langle \p_{k_1}^{(N)} \p_{k_2}^{(N)} \rangle - \langle \p_{k_1}^{(N)} \rangle
		 \langle \p_{k_2}^{(N)} \rangle,
	\end{equation*}	
since these are what we need to understand in order to obtain
Theorem \ref{thm:QuantumLLN}, the Law of Large Numbers for BPP observables of
the LR process.

\subsection{Covariance setup}
According to Biane's formula \eqref{eqn:QuantumCorrelators}, the covariance of the 
classical random variables $\p_{k_1}^{(N)},\p_{k_2}^{(N)}$ coincides 
with the covariance of the quantum random variables $\tr(C_N^{k_1}),\tr(C_N^{k_2})$:
	\begin{equation*}
		\langle \p_{k_1}^{(N)}\p_{k_2}^{(N)} \rangle_c =
		\E[\tr(C_N^{k_1})\tr(C_N^{k_2})] - \E[\tr(C_N^{k_1})]\ \E\tr[(C_N^{k_2})]
	\end{equation*}
for any $k_1,k_2 \in \N^*$, where $C_N=A_N+B_N$.  For example, in the case of the 
simplest connected $2$-point function, the variance of $\p_1^{(N)}$, we have 
	\begin{align*}
		\langle \p_1^{(N)}\p_1^{(N)} \rangle_c &=\E[\tr(C_N)\tr(C_N)] - \E[\tr(C_N)]\E[\tr(C_N)] \\
		&= \E[\tr(A_N)\tr(A_N)] - \E[\tr(A_N)]\ \E[\tr(A_N)] \\
		&+ \E[\tr(A_N)\tr(B_N)] - \E[\tr(A_N)]\ \E[\tr(B_N)] \\
		&+ \E[\tr(B_N)\tr(A_N)] - \E[\tr(B_N)]\ \E[\tr(A_N)] \\
		&+ \E[\tr(B_N)\tr(B_N)] - \E[\tr(B_N)]\ \E[\tr(B_N)].
	\end{align*}
In general, in order to compute $\langle \p_{k_1}^{(N)}\p_{k_2}^{(N)} \rangle_c$ in the
semiclassical/large-dimension limit, we need to be able to estimate differences of the form
	\begin{align*}
		&\E\left[\tr(A_N^{p_1(1)}B_N^{q_1(1)} \cdots A_N^{p_1(d_1)}B_N^{q_1(d_1)})
		\tr(A_N^{p_2(1)}B_N^{q_2(1)} \cdots A_N^{p_2(d_2)}B_N^{q_2(d_2)})\right] \\
			-&\E\left[\tr(A_N^{p_1(1)}B_N^{q_1(1)} \cdots A_N^{p_1(d_1)}B_N^{q_1(d_1)})\right]\ 
			\E\left[\tr(A_N^{p_2(1)}B_N^{q_2(1)} \cdots A_N^{p_2(d_2)}B_N^{q_2(d_2)})\right]
	\end{align*}
in the semiclassical/large-dimension limit, where $d_1,d_2 \in \N^*$ are fixed positive integers
and $p_1,q_1 \colon [d_1] \to \N$ and $p_2,q_d \colon [d_2] \to \N$ are fixed functions.
Let us write 
	\begin{equation*}
		\tau_{12}^{(N)} := \E\left[\tr(A_N^{p_1(1)}B_N^{q_1(1)} \cdots 
		A_N^{p_1(d_1)}B_N^{q_1(d_1)})\tr(A_N^{p_2(1)}B_N^{q_2(1)} \cdots 
		A_N^{p_2(d_2)}B_N^{q_2(d_2)})\right]
	\end{equation*}
and 
	\begin{align*}
		\tau_1^{(N)} &:= \E\tr(A_N^{p_1(1)}B_N^{q_1(1)} \cdots A_N^{p_1(d_1)}B_N^{q_1(d_1)}) \\
		\tau_2^{(N)} &:= \E\tr(A_N^{p_2(1)}B_N^{q_2(1)} \cdots A_N^{p_2(d_2)}B_N^{q_2(d_2)})
	\end{align*}	
for these quantities, so that our goal is to estimate the difference 
	\begin{equation*}
		\tau_{12}^{(N)} - \tau_1^{(N)}\tau_2^{(N)}
	\end{equation*}
in the semiclassical/large-dimension limit.  In view of our analysis of mean values in 
Sections \ref{sec:MeanValuesPlanckScale} and \ref{sec:MeanValueAsymptotics}, 
the second term is understood.
It remains to analyze the first term, $\tau_{12}^{(N)}$, and in order to do this we will
generalize the approach developed above.

\subsection{$2$-point functions at the Planck scale}
Let us rewrite $\tau_{12}^{(N)}$ as follows.  Put $d=d_1+d_2$, and define functions
$p,q\colon [d] \to \N$ by
	\begin{align*}
		p|_{[d_1]}&=p_1,& p|_{[d_1+1,d]}&=p_2 \\
		q|_{[d_1]}&=q_1,& q|_{[d_1+1,d]}&=q_2.
	\end{align*}
We then have
	\begin{multline*}
		\tau_{12}^{(N)} = \E\left[
		\tr(A_N^{p(1)}B_N^{q(1)} \cdots A_N^{p(d_1)}B_N^{q(d_1)}) \right. \\ \left.
		\tr(A_N^{p(d_1+1)}B_N^{q(d_1+1)} \cdots A_N^{p(d_1+d_2)}B_N^{q(d_1+d_2)})
		\right].
	\end{multline*}
It will be convenient to affect the same notational change for the 
quantities $\tau_1^{(N)}$ and $\tau_2^{(N)}$, that is we write
	\begin{align*}
		\tau_1^{(N)} &= \E\tr(A_N^{p(1)}B_N^{q(1)} \cdots A_N^{p(d_1)}B_N^{q(d_1)}), \\
		\tau_2^{(N)} &= \E\tr(A_N^{p(d_1+1)}B_N^{q(d_1+1)} \cdots 
		A_N^{p(d_1+d_2)}B_N^{q(d_1+d_2)}).
	\end{align*}
	
We now analyze $\tau_{12}^{(N)}$ following the same steps as in 
\cref{sec:MeanValuesPlanckScale,sec:MeanValueAsymptotics}.

\subsubsection{Unitary invariance}

	\begin{proposition}
		Define a function $f \colon \group{U}(N) \rightarrow \C$ by 
			\begin{multline*}
				f_N(U) := \E\left[
		\tr\left(UA_N^{p(1)}U^{-1}B_N^{q(1)} \cdots UA_N^{p(d_1)}U^{-1}B_N^{q(d_1)}\right) \right.
                \times \\ \times 
                \left.
		\tr\left(UA_N^{p(d_1+1)}U^{-1}B_N^{q(d_1+1)} \cdots 
		UA_N^{p(d_1+d_2)}U^{-1}B_N^{q(d_1+d_2)}\right)
		\right].
			\end{multline*}
		Then, $f_N$ is constant, being equal to $\tau_{12}^{(N)}$ for all $U 
		\in \group{U}(N)$.	
	\end{proposition}
	
As a consequence of this invariance, we have
	\begin{equation*}
		\tau_{12}^{(N)}= \int_{\group{U}(N)} f_N(U) \mathrm{d}U.
	\end{equation*}
We want to use this in exactly the same way as we did in our mean value computation.

Expanding the first trace yields the sum
	\begin{multline*}
		\frac{1}{N}\sum_{r_1\colon [4d_1] \to [N]}
		U_{r_1(1)r_1(2)}\left(A_N^{p(1)}\right)_{r_1(2)r_1(3)}U^{-1}_{r_1(3)r_1(4)}
		 \left(B_N^{q(1)}\right)_{r_1(4)r_1(5)} 
		\cdots  \\ 
		\cdots U_{r_1(4d_1-3)r_1(4d_1-2)}\left(A_N^{p(d_1)}\right)_{r_1(4d_1-2)r_1(4d_1-1)}
		U^{-1}_{r_1(4d_1-1)r_1(4d_1)} \left(B_N^{q(d_1)}\right)_{r_1(4d_1)r_1(1)}.
	\end{multline*}
Expanding the second trace yields the sum
	\begin{multline*}
		\frac{1}{N}\sum_{r_2\colon [4d_2] \to [N]}
		U_{r_2(1)r_2(2)}\left(A_N^{p(d_1+1)}\right)_{r_2(2)r_2(3)}U^{-1}_{r_2(3)r_2(4)}
		 \left(B_N^{q(d_1+1)}\right)_{r_2(4)r_2(5)} 
		\cdots  \\ 
		\cdots U_{r_2(4d_2-3)r_2(4d_2-2)}\left(A_N^{p(d_1+d_2)}\right)_{r_2(4d_2-2)r_2(4d_2-1)}
		U^{-1}_{r_2(4d_2-1)r_2(4d_2)} \left(B_N^{q(d_1+d_2)}\right)_{r_2(4d_2)r_2(1)}.
	\end{multline*}

For the first trace, let us reparametrize the summation index $r_1\colon [4d_1] \to [N]$ 
by a quadruple of functions $i_1,j_1,i_1',j_1'\colon [d_1] \to [N]$ according to
	\begin{multline*}
	\big(r_1(1),r_1(2),r_1(3),r_1(4),\dots,r_1(4d_1-3),r_1(4d_1-2),r_1(4d_1-1),r_1(4d_1)\big) \\
	= \big(i_1(1),j_1(1),j_1'(1),i_1'(1),\dots,i_1(d_1),j_1(d_1),j_1'(d_1),i_1'(d_1)\big).
		\end{multline*}
Then, the above expansion of the first trace becomes
	\begin{align*}
		&\frac{1}{N}\sum_{i_1,j_1,_1i',j_1'\colon [d_1] \to [N]} L_N(i_1,j_1,i_1',j_1')\ 
		\prod_{k=1}^{d_1}
				\left(A_N^{p(k)}\right)_{j_1(k)j_1'(k)} 
				\left(B_N^{q(k)}\right)_{i_1'(k)i_1\gamma_1(k)} \\
		=& \frac{1}{N}\sum_{i_1,j_1,i_1',j_1'\colon [d_1] \to [N]} L_N(i_1,j_1,i_1',j_1')\ 
		\prod_{k=1}^{d_1}
				\left( A_N^{p(k)}\right)_{j_1(k)j_1'(k)}
			\prod_{k=1}^{d_1} \left(B_N^{q(k)}\right)_{i_1'(k)i_1\gamma_1(k)},
	\end{align*}
where
	\begin{equation*}
		L_N(i_1,j_1,i_1',j_1') = \prod_{k=1}^{d_1} U_{i_1(k)j_1(k)} \overline{U}_{i_1'(k)j_1'(k)},
	\end{equation*}
and $\gamma_1$ is the cycle $(1\ 2\ \dots\ d_1)$ in $\Sy{d_1+d_2}$.  

Similarly, if we reparametrize the summation index $r_2\colon [4d_2] \to [N]$ 
by a quadruple of functions $i_2,j_2,i_2',j_2'\colon [d_2] \to [N]$ according to
	\begin{multline*}
		\big(r_2(1),r_2(2),r_2(3),r_2(4),\dots,r_2(4d_2-3),r_2(4d_2-2),r_2(4d_2-1),r_2(4d_2)\big) \\
	= \big(i_2(d_1+1),j_2(d_1+1),j_2'(d_1+1),i_2'(d_1+1),\dots,
	i_2(d),j_2(d),j_2'(d),i_2'(d)\big),
		\end{multline*}
the expansion of the second trace takes the form
\begin{align*}
		&\frac{1}{N}\sum_{i_2,j_2,i_2',j_2'\colon [d_1+1,d_1+d_2] \to [N]} L_N(i_2,j_2,i_2',j_2')\ 
		\prod_{k=d_1+1}^{d_1+d_2}
				\left(A_N^{p(k)}\right)_{j_2(k)j_2'(k)} 
				\left(B_N^{q(k)}\right)_{i_2'(k)i_2\gamma_2'(k)} \\
		=& \frac{1}{N}\sum_{i_2,j_2,i_2',j_2'\colon [d_1+1,d_1+d_2] \to [N]} L_N(i_2,j_2,i_2',j_2')\ 
		\prod_{k=d_1+1}^{d_1+d_2}
				\left( A_N^{p(k)}\right)_{j_2(k)j_2'(k)}
			\prod_{k=1}^{d_2} \left(B_N^{q(k)}\right)_{i_2'(k)i_2\gamma_2'(k)},
	\end{align*}
where
	\begin{equation*}
		L_N(i_2,j_2,i_2',j_2') = \prod_{k=d_1+1}^{d_1+d_2} 
		U_{i_2(k)j_2(k)} \overline{U}_{i_2'(k)j_2'(k)},
	\end{equation*}
and $\gamma_2'$ is the cycle $(d_1+1\ d_1+2\ \dots\ d_1+d_2)$ in $\Sy{d_1+d_2}$.
	
We now smash the expansions of the two traces together to get the huge compound expansion
	\begin{align*}
		&\frac{1}{N^2}\sum_{i_1,j_1,_1i',j_1'\colon [d_1] \to [N]}\ 
		\sum_{i_2,j_2,i_2',j_2'\colon [d_1+1,d_1+d_2] \to [N]}
		L_N(i_1,j_1,i_1',j_1')\ L_N(i_2,j_2,i_2',j_2') \\
		&\times \prod_{k=1}^{d_1}\left( A_N^{p(k)}\right)_{j_1(k)j_1'(k)}
		\prod_{k=d_1+1}^{d_1+d_2}\left( A_N^{p(k)}\right)_{j_2(k)j_2'(k)}
		\prod_{k=1}^{d_1} \left(B_N^{q(k)}\right)_{i_1'(k)i_1\gamma_1(k)} 
		\prod_{k=d_1+1}^{d_1+d_2} \left(B_N^{q(k)}\right)_{i_2'(k)i_2\gamma_2'(k)} \\
		=&\frac{1}{N^2}\sum_{i,j,i',j'\colon [d_1+d_2] \to [N]}\ 
		L_N(i,j,i',j')
		\prod_{k=1}^{d_1+d_2} \left( A_N^{p(k)}\right)_{j(k)j'(k)}
		\prod_{k=1}^{d_1+d_2} \left(B_N^{q(k)}\right)_{i'(k)i\gamma_1\gamma_2'(k)},
	\end{align*}
where
	\begin{equation*}
		L_N(i,j,i',j')=\prod_{k=1}^{d_1+d_2} U_{i(k)j(k)}\ \overline{U}_{i'(k)j'(k)}.
	\end{equation*}
Thus, we obtain the following representation of $\tau_{12}^{(N)}$:
	\begin{multline*}
		\tau_{12}^{(N)} = \frac{1}{N^2}\sum_{i,j,i',j'\colon [d_1+d_2] \to [N]}\ 
		I_N(i,j,i',j') \times \\ \times
		\E\left[\prod_{k=1}^{d_1+d_2} \left( A_N^{p(k)}\right)_{j(k)j'(k)}\right]
		\E\left[\prod_{k=1}^{d_1+d_2} \left(B_N^{q(k)}\right)_{i'(k)i\gamma_1\gamma_2'(k)}\right],
	\end{multline*}
where 
	\begin{equation*}
		I_N(i,j,i',j')=\int_{\group{U}(N)} L_N(i,j,i',j') \mathrm{d}U.
	\end{equation*}
Plugging \eqref{eqn:WeingartenConvolutionFormula} into our calculation above 
eliminates the indices $i',j'$ and produces the formula
	\begin{equation*}
		\tau_{12}^{(N)}=\frac{1}{N^2} \sum_{({\pi_1},{\pi_2}) \in \Sy{d_1+d_2}^2} \Wg_N({\pi_1},{\pi_2})\
 		\E[S^A_{{\pi_1}}]\ \E[S^B_{{\pi_2}^{-1}\gamma_1\gamma_2'}],
	\end{equation*}
where
	\begin{align*}
		S^A_{{\pi_1}}
		:=& \sum_{i\colon [d_1+d_2] \to [N]} \prod_{k=1}^{d_1+d_2}
		\left(A_N^{p(k)}\right)_{i(k)i{\pi_1}(k)} \\
	\intertext{and} 
		S^B_{{\pi_2}^{-1}\gamma_1\gamma_2'} :=&
		\sum_{i\colon [d_1+d_2] \to [N]} \prod_{k=1}^{d_1+d_2} 
		\left(B_N^{q(k)}\right)_{i(k)i{\pi_2}^{-1}\gamma_1
		\gamma_2'(k)}.
	\end{align*} 
This generalizes \eqref{eq:here-we-use-classical-independence} 
from the expected trace to the expected
product of two traces; the formula could be generalized to the expected product of any
finite number of traces in just the same way.

Our goal now is to express the operators $S^A_{{\pi_1}}$ and $S^B_{{\pi_2}^{-1}\gamma_1\gamma_2'}$
in terms of the operators
	\begin{equation*}
		\tr A_N, \tr A_N^2, \dots \quad\text{ and } \tr B_N, \tr B_N^2, \dots.
	\end{equation*}
Just like in the mean value computation, we lift the problem to the universal enveloping algebra
and use Biasimirs.   

\subsubsection{Biasimirs again}
The operators $S^A_{\pi_1}$ and $S^B_{{\pi_2}^{-1}\gamma_1\gamma_2'}$ are, up to tensoring with 
an identity operator, images of Biasimirs in irreducible representations:
	\begin{align*}
			S^A_{{\pi_1}}
				&= \rho_N(C_{{\pi_1}}^{(p)}) \otimes I_{\W_N} \\
			S^B_{{\pi_2}^{-1}\gamma_1\gamma_2'}
				&= I_{\V_N} \otimes \sigma_N(C_{{\pi_2}^{-1}\gamma_1\gamma_2'}^{(q)}).
	\end{align*}
Each of the Biasimirs $C_{\pi_1}^{(p)}$ and $C_{{\pi_2}^{-1}\gamma_1\gamma_2'}^{(q)}$ has
its own classical/quantum decomposition:
	\begin{align*}
		C_{\pi_1}^{(p)} &= \mathbf{P}_{\pi_1}^{(p)}(C_1,\dots,C_{|p|}) + 
			\hbar_N \mathbf{Q}_{\pi_1}^{(p)}(\hbar_N,N,C_1,\dots,C_{|p|}) \\
		C_{{\pi_2}^{-1}\gamma_1\gamma_2'}^{(q)} &=
			\mathbf{P}_{{\pi_2}^{-1}\gamma_1\gamma_2'}^{(q)}(C_1,\dots,C_{|q|})
			+\hbar_N \mathbf{Q}_{{\pi_2}^{-1}\gamma_1\gamma_2'}^{(q)}(\hbar_N,N,C_1,
			\dots,C_{|q|}).
	\end{align*}
	
Now we come back to the operators $S^A_{\pi_1}$ and $S^B_{{\pi_2}^{-1}\gamma_1\gamma_2'}$.
First, we have that
	\begin{equation*}
		S^A_{\pi_1} = \mathbf{P}_{\pi_1}^{(p)}(\Tr A_N,\dots,\Tr A_N^{|p|}) +
		\hbar_N \mathbf{Q}_{\pi_1}^{(p)}(\hbar_N,N,\Tr A_N,\dots,\Tr A_N^{|p|}).
	\end{equation*}
Let us rewrite this in terms of normalized traces.  Put
	\begin{align*}
		\PPA{{\pi_1}} &:=\frac{1}{N^{\cyc({\pi_1})}} 
			\mathbf{P}_{\pi_1}^{(p)}\left( \Tr A_N,\dots,\Tr A_N^{|p|}\right), \\
		\QQA{{\pi_1}} &:=\frac{1}{N^{\aex({\pi_1})}} 
                             \mathbf{Q}_{\pi_1}^{(p)}\left(\hbar_N,N,\Tr A_N,\dots,\Tr A_N^{|p|}\right).
           \end{align*}
$\PPA{{\pi_1}}$ is an
explicit polynomial in the operators
\[ \tr A_N, \tr A_N^2,\dots, \tr A_N^{|p|}, \]
while $\QQA{{\pi_1}}$ is a polynomial in the 
numbers $\hbar_N,N^{-1}$ and the operators
\begin{equation*}
 \tr A_N, \tr A_N^2,\dots, \tr A_N^{|p|}.
\end{equation*}
We thus have
	\begin{equation*}
		S^A_{\pi_1} = N^{\cyc({\pi_1})}\PPA{{\pi_1}} + \hbar_N N^{\aex({\pi_1})}\QQA{{\pi_1}}.
	\end{equation*}
Now we want to apply the expectation $\E$ to both sides of this identity in $\A_N$ to 
get an identity in $\C$.  We set
	\begin{align*}
		\PPA{{\pi_1}}(N)&:=\E\PPA{{\pi_1}} =\frac{1}{N^{\cyc({\pi_1})}} 
			\mathbf{P}_{\pi_1}^{(p)}\left( \E\Tr A_N,\dots,\E\Tr A_N^{|p|}\right), \\
		\QQA{{\pi_1}}(N)&:=\E\QQA{{\pi_1}} =\frac{1}{N^{\aex({\pi_1})}} 
                             \mathbf{Q}_{\pi_1}^{(p)}\left(\hbar_N,N,\E\Tr A_N,\dots,\E\Tr A_N^{|p|}\right);
           \end{align*}
the first of these is a polynomial in the numbers
\[ \E\tr A_N, \E\tr A_N^2,\dots, \E\tr A_N^{|p|}, \]
while the second is a polynomial in the numbers 
\begin{equation*}
 \hbar_N,N^{-1},\E\tr A_N, \E\tr A_N^2,\dots, \E\tr A_N^{|p|}.
\end{equation*}
We conclude that
	\begin{equation*}
		\E[S^A_{\pi_1}] = N^{\cyc({\pi_1})}\PPA{{\pi_1}}(N) + \hbar_N N^{\aex({\pi_1})}\QQA{{\pi_1}}(N).
	\end{equation*}

\bigskip
	
Second, we have that
	\begin{multline*}
		S^B_{{\pi_2}^{-1}\gamma_1\gamma_2'} 
		= \mathbf{P}_{{\pi_2}^{-1}\gamma_1\gamma_2'}^{(q)}(\Tr B_N,\dots,\Tr B_N^{|q|}) + \\ +
		\hbar_N \mathbf{Q}^{(q)}_{{\pi_2}^{-1}\gamma_1\gamma_2'}(\hbar_N,N,\Tr B_N,\dots,
		\Tr B_N^{|q|}).
	\end{multline*}
Once again, let us rewrite this in terms of normalized traces.  Put
	\begin{align*}
		\PPB{{\pi_2}^{-1}\gamma_1\gamma_2'} &:=\frac{1}{N^{\cyc({\pi_2}^{-1}\gamma_1\gamma_2')}} 
			\mathbf{P}_{{\pi_2}^{-1}\gamma_1\gamma_2'}^{(q)}\left( \Tr B_N,\dots,
			\Tr B_N^{|q|}\right), \\
		\QQB{{\pi_2}^{-1}\gamma_1\gamma_2'} &:=\frac{1}{N^{\aex({\pi_2}^{-1}\gamma_1\gamma_2')}} 
                             \mathbf{Q}_{{\pi_2}^{-1}\gamma_1\gamma_2'}^{(q)}
                             \left(\hbar_N,N,\Tr B_N,\dots,\Tr B_N^{|q|}\right).
           \end{align*}
$\PPB{{\pi_2}^{-1}\gamma_1\gamma_2'}$ 
is an explicit polynomial in the operators
\[ \tr B_N, \tr B_N^2,\dots, \tr B_N^{|q|}, \]
while $\QQB{{\pi_2}^{-1}\gamma_1\gamma_2'}$ 
is a polynomial in the 
numbers $\hbar_N,N^{-1}$ and the operators
\begin{equation*}
\label{eq:useful-quantites-B2}
 \tr B_N, \tr B_N^2,\dots, \tr B_N^{|q|}.
\end{equation*}
We thus have
	\begin{equation*}
		S^B_{{\pi_2}^{-1}\gamma_1\gamma_2'} = N^{\cyc({\pi_2}^{-1}\gamma_1\gamma_2')}
		\PPB{{\pi_2}^{-1}\gamma_1\gamma_2'} + 
		\hbar_N N^{\aex({\pi_2}^{-1}\gamma_1\gamma_2')}\QQB{{\pi_2}^{-1}\gamma_1\gamma_2'}.
	\end{equation*}
We apply $\E$ to both sides of this identity in $\A_N$ to 
get an identity in $\C$.  As above, we declare
	\begin{align*}
		\PPB{{\pi_2}^{-1}\gamma_1\gamma_2'}(N)&:=\E\PPB{{\pi_2}^{-1}\gamma_1\gamma_2'} = \\ & \hspace{5ex}
				\frac{1}{N^{\cyc({\pi_2}^{-1}\gamma_1\gamma_2')}} 
			\mathbf{P}_{{\pi_2}^{-1}\gamma_1\gamma_2'}^{(q)}
			\left( \E\Tr B_N,\dots,\E\Tr B_N^{|q|}\right), \\
		\QQB{{\pi_2}^{-1}\gamma_1\gamma_2'}(N)&:=\E\QQB{{\pi_2}^{-1}\gamma_1\gamma_2'} = \\ & 
			\frac{1}{N^{\aex({\pi_2}^{-1}\gamma_1\gamma_2')}} 
                             \mathbf{Q}_{{\pi_2}^{-1}\gamma_1\gamma_2'}^{(q)}
                             \left(\hbar_N,N,\E\Tr B_N,\dots,\E\Tr B_N^{|q|}\right).
           \end{align*}
The first of these is a polynomial in the numbers
\[ \E\tr B_N, \E\tr B_N^2,\dots, \E\tr B_N^{|q|}, \]
while the second is a polynomial in the numbers 
\begin{equation*}
 \hbar_N,N^{-1},\E\tr B_N, \E\tr B_N^2,\dots, \E\tr B_N^{|q|}.
\end{equation*}
We conclude that
	\begin{multline*}
		\E[S^B_{{\pi_2}^{-1}\gamma_1\gamma_2'}] = \\ 
		N^{\cyc({\pi_2}^{-1}\gamma_1\gamma_2')}\ \PPB{{\pi_2}^{-1}\gamma_1\gamma_2'}(N) 
		+ \hbar_N N^{\aex({\pi_2}^{-1}\gamma_1\gamma_2')}\ \QQB{{\pi_2}^{-1}\gamma_1\gamma_2'}(N).
	\end{multline*}
	
\subsubsection{Classical/Quantum Decomposition of $\tau_{12}^{(N)}$}
Putting these two calculations together, we compute the $({\pi_1},{\pi_2})$ term of 
$\tau_{12}^{(N)}$ as
	\begin{multline*}
		\Wg_N({\pi_1},{\pi_2})\ \E[S^A_{\pi_1}]\ \E[S^B_{{\pi_2}^{-1}\gamma_1\gamma_2'}] 
		= \\ \Wg_N({\pi_1},{\pi_2})
			\left( N^{\cyc({\pi_1})}\PPA{{\pi_1}}(N) + 
			\hbar_N N^{\aex({\pi_1})}\QQA{{\pi_1}}(N) \right) \times \\
			\times \left( N^{\cyc({\pi_2}^{-1}\gamma_1\gamma_2')}
			\PPB{{\pi_2}^{-1}\gamma_1\gamma_2'}(N) 
				+ \hbar_N N^{\aex({\pi_2}^{-1}\gamma_1\gamma_2')}
				\QQB{{\pi_2}^{-1}\gamma_1\gamma_2'}(N) \right).
	\end{multline*}        
Expanding the brackets and summing $({\pi_1},{\pi_2})$ over $\Sy{d}^2$, we arrive at the 
classical/quantum decomposition of $\tau_{12}^{(N)}$.
	
	\begin{theorem}
	\label{thm:TwoPointClassicalQuantumDecomposition}
	We have
		\begin{equation*}
			\tau_{12}^{(N)} = \Classical_{12}^{(N)} + \hbar_N \Quantum_{12}^{(N)},
		\end{equation*}
	where
		\begin{multline*}
			\Classical_{12}^{(N)} =  \frac{1}{N^2}\sum_{({\pi_1},{\pi_2}) \in \Sy{d}^2} 
			N^{\cyc({\pi_1})+\cyc({\pi_2}^{-1}\gamma_1\gamma_2')}\times \\ \times \Wg_N({\pi_1},{\pi_2})\
			\PPA{{\pi_1}}(N)\ \PPB{{\pi_2}^{-1}\gamma_1\gamma_2'}(N)
		\end{multline*}
	and
		\begin{multline*}
			\Quantum_{12}^{(N)} =\\ \begin{aligned}\hspace{5ex} & \frac{1}{N^2}\sum_{({\pi_1},{\pi_2}) \in \Sy{d}^2} 
			\bigg{(}
			N^{\cyc({\pi_1})+\aex({\pi_2}^{-1}\gamma_1\gamma_2')} \Wg_N({\pi_1},{\pi_2})
				\PPA{{\pi_1}}(N) \QQB{{\pi_2}^{-1}\gamma_1\gamma_2'}(N) \\
			&+
				N^{\aex({\pi_1})+\cyc({\pi_2}^{-1}\gamma_1\gamma_2')}\Wg_N({\pi_1},{\pi_2})
				\QQA{{\pi_1}}(N)\PPB{{\pi_2}^{-1}\gamma_1\gamma_2'}(N) \\
			&+  \hbar_N N^{\aex({\pi_1})+\aex({\pi_2}^{-1}\gamma_1\gamma_2')}\Wg_N({\pi_1},{\pi_2})
				\QQA{{\pi_1}}(N)\QQB{{\pi_2}^{-1}\gamma_1\gamma_2'}(N)
			\bigg{)}.
		\end{aligned} \end{multline*}
	\end{theorem}
	
\subsection{Covariance asymptotics}
We now apply the above exact computations to obtain the semiclassical asymptotics
of $\tau_{12}^{(N)}$.  The analysis is a direct generalization of the mean value asymptotic
analysis carried out in \cref{sec:MeanValueAsymptotics}.  As in  
\cref{sec:MeanValueAsymptotics}, we work under the hypotheses of \cref{thm:Main}.

\subsubsection{Quantum asymptotics}

The following combinatorial result is a reformulation 
of the result of Biane \cite[page 173]{Biane98}, see \cref{sec:proof-of-lem:Defect} 
for a detailed corrected proof.
\begin{lemma}
\label{lem:SecondDefect} 
For any permutations ${\pi_1},{\pi_2} \in \Sy{d}$, we have
			\begin{equation}
				\label{eq:defects-nice2}
				\defect({\pi_1})+\defect({\pi_2}^{-1}\gamma_1\gamma_2') \leq
				2\genus({\pi_1},{\pi_1}^{-1}{\pi_2},{\pi_2}^{-1}\gamma_1\gamma_2').
			\end{equation}
\end{lemma}

	\begin{theorem}
		For any bounded $\hbar_N=O(1)$, 
                the quantum part
		$\Quantum_{12}^{(N)}$ of $\tau_{12}^{(N)}$ is $O(1)$ as $N \rightarrow \infty$.
	\end{theorem}
	\begin{proof}
		The quantum part of $\tau_{12}^{(N)}$ may be 
		written as 
			\begin{multline}
                         \label{eq:12quantumsum}
			\Quantum^{(N)}_{12} = \\ \sum_{({\pi_1},{\pi_2}) \in \Sy{d}^2} 
			N^{\cyc({\pi_1})+\cyc({\pi_2}^{-1}\gamma_1\gamma_2')-2} \Wg_N({\pi_1},{\pi_2})\
			\mathbf{R}_{({\pi_1},{\pi_2})}(N)
			\end{multline}
		where 
			\begin{align*}
			\mathbf{R}_{({\pi_1},{\pi_2})}(N) =&
			N^{\defect({\pi_2}^{-1}\gamma_1\gamma_2')}
				\ \PPA{{\pi_1}}(N)\ \QQB{{\pi_2}^{-1}\gamma_1\gamma_2'}(N) \\
				&+N^{\defect({\pi_1})}\
				\QQA{{\pi_1}}(N)\ \PPB{{\pi_2}^{-1}\gamma_1\gamma_2'}(N)\\
			&+\hbar_N N^{\defect({\pi_1})+\defect({\pi_2}^{-1}\gamma_1\gamma_2')}
				\ \QQA{{\pi_1}}(N)\ \QQB{{\pi_2}^{-1}\gamma_1\gamma_2'}(N).
			\end{align*}
		We will show that each term in the sum \eqref{eq:12quantumsum} is $O(1)$.	
		
		By the first part of \cref{lem:Defect}, nonnegativity of the defect statistic,
                and \cref{prop:PPA-QQA-asymptoticalyO1}
		we have
			\begin{equation*}
				\mathbf{R}_{({\pi_1},{\pi_2})}(N) = 
				O\left(N^{\defect({\pi_1})+\defect({\pi_2}^{-1}\gamma_1\gamma_2')}\right)
			\end{equation*}
		for each $({\pi_1},{\pi_2}) \in \Sy{d}^2$. 
		
		Consider now the asymptotics of the Weingarten factor,
			\begin{equation*}
			N^{\cyc({\pi_1})+\cyc({\pi_2}^{-1}\gamma_1\gamma_2')-2} \Wg_N({\pi_1},{\pi_2}).
			\end{equation*}
		By \eqref{eqn:WeingartenReduced}, we have
		\begin{multline*}
			N^{\cyc({\pi_1})+\cyc({\pi_2}^{-1}\gamma_1\gamma_2')-2}\Wg_N({\pi_1},{\pi_2})
			\\ \begin{aligned}  &= N^{\cyc({\pi_1})+\cyc({\pi_2}^{-1}\gamma_1\gamma_2')-2} 
				\frac{(-1)^{|{\pi_1}^{-1}{\pi_2}|}}{N^{d+|{\pi_1}^{-1}{\pi_2}|}}
				\sum_{g=0}^{\infty} \frac{\vec{W}_g({\pi_1},{\pi_2})}{N^{2g}} \\
		&= N^{-|{\pi_1}|-|{\pi_1}^{-1}{\pi_2}|-|{\pi_2}^{-1}\gamma_1\gamma_2'| + 
		|\gamma_1\gamma_2'|}\ 
		(-1)^{|{\pi_1}^{-1}{\pi_2}|}\sum_{g=0}^{\infty} \frac{\vec{W}_g({\pi_1},{\pi_2})}{N^{2g}} \\
		&= (-1)^{|{\pi_1}^{-1}{\pi_2}|} \
			N^{-2\genus({\pi_1},{\pi_1}^{-1}{\pi_2},{\pi_2}^{-1}\gamma_1\gamma_2')}
			\sum_{g=0}^{\infty} \frac{\vec{W}_g({\pi_1},{\pi_2})}{N^{2g}}.
		    \end{aligned}
		\end{multline*}
		Thus,
			\begin{equation*}
				N^{\cyc({\pi_1})+\cyc({\pi_2}^{-1}\gamma_1\gamma_2')-2} \Wg_N({\pi_1},{\pi_2})
				= O(N^{-2\genus({\pi_1},{\pi_1}^{-1}{\pi_2},{\pi_2}^{-1}\gamma_1\gamma_2')}).
			\end{equation*}
			
		We conclude that the order of the $({\pi_1},{\pi_2})$ term in the sum is
			\begin{equation*}
				O\left(N^{\defect({\pi_1})+\defect({\pi_2}^{-1}\gamma_1\gamma_2') -
				2\genus({\pi_1},{\pi_1}^{-1}{\pi_2},{\pi_2}^{-1}\gamma_1\gamma_2')}\right).
			\end{equation*}
		By \cref{lem:SecondDefect}, the exponent
			\begin{equation*}
				\defect({\pi_1})+\defect({\pi_2}^{-1}\gamma_1\gamma_2') -
				2\genus({\pi_1},{\pi_1}^{-1}{\pi_2},{\pi_2}^{-1}\gamma_1\gamma_2')
			\end{equation*}
		is nonpositive; consequently, each term of $\Quantum_{12}^{(N)}$ is $O(1)$,
		as required.
	\end{proof}

\subsubsection{Classical asymptotics}
	\begin{theorem}
		For each $N \geq d$,
		the classical part of $\tau_{12}^{(N)}$ admits, for each $N \geq d$,
		an absolutely convergent series expansion of the form
			\begin{equation*}
				\Classical_{12}^{(N)} = \sum_{k=0}^{\infty} \frac{e_k^{(12)}(N)}{N^{2k}},
			\end{equation*}
		the coefficients of which are given by 
			\begin{multline*}
				e_k^{(12)}(N) := \sum_{\substack{(g,h) \in \N^2 \\ g+h=k}} 
				\sum_{\substack{({\pi_1},{\pi_2}) \in \Sy{d}^2 \\
				\genus({\pi_1},{\pi_1}^{-1}{\pi_2},{\pi_2}^{-1}\gamma_1\gamma_2')=h}}
				(-1)^{|{\pi_1}^{-1}{\pi_2}|} \times \\ \times \vec{W}_g({\pi_1},{\pi_2})\
				\PPA{{\pi_1}}(N)\ \PPB{{\pi_2}^{-1}\gamma_1\gamma_2'}(N).
			\end{multline*}
	\end{theorem}
	
	\begin{proof}
		According to \cref{thm:TwoPointClassicalQuantumDecomposition}
		and the expansion \eqref{eqn:WeingartenReduced}, we have
\begin{align*}
				\Classical_{12}^{(N)} &= 
\sum_{({\pi_1},{\pi_2}) \in \Sy{d}^2} 
				N^{\cyc({\pi_1})+\cyc({\pi_2}^{-1}\gamma)-2} 
                                \times \\[-2ex] & \hspace{15ex} \times 
                                \Wg_N({\pi_1},{\pi_2})\ 
				\PPA{{\pi_1}}(N)\ \PPB{{\pi_2}^{-1}\gamma_1\gamma_2'}(N) \\[2ex]
				&= \sum_{({\pi_1},{\pi_2}) \in \Sy{d}^2} 
				N^{-2\genus({\pi_1},{\pi_1}^{-1}{\pi_2},{\pi_2}^{-1}\gamma_1\gamma_2')}
				(-1)^{|{\pi_1}^{-1}{\pi_2}|} \times \\[-2ex] & \hspace{15ex} \times
				\PPA{{\pi_1}}(N)\ \PPB{{\pi_2}^{-1}\gamma_1\gamma_2'}(N)
				\sum_{g=0}^{\infty} \frac{\vec{W}_g({\pi_1},{\pi_2})}{N^{2g}} \\[2ex]
				&=\sum_{g,h=0}^{\infty} \frac{1}{N^{2(g+h)}}
				\sum_{\substack{({\pi_1},{\pi_2}) \in \Sy{d}^2 \\
				\genus({\pi_1},{\pi_1}^{-1}{\pi_2},{\pi_2}^{-1}\gamma_1\gamma_2')=h}}
				(-1)^{|{\pi_1}^{-1}{\pi_2}|}
				\times \\ & \hspace{15ex} \times
				\vec{W}_g({\pi_1},{\pi_2})\
				\PPA{{\pi_1}}(N)\ \PPB{{\pi_2}^{-1}\gamma_1\gamma_2'}(N) \\[2ex]
				&= \sum_{k=0}^{\infty} \sum_{\substack{g,h\geq 0\\ g+h=k}} \frac{1}{N^k}
				\sum_{\substack{({\pi_1},{\pi_2}) \in \Sy{d}^2 \\
				\genus({\pi_1},{\pi_1}^{-1}{\pi_2},{\pi_2}^{-1}\gamma_1\gamma_2')=h}}
				(-1)^{|{\pi_1}^{-1}{\pi_2}|}
				\times \\ & \hspace{15ex} \times
				\vec{W}_g({\pi_1},{\pi_2})\
				\PPA{{\pi_1}}(N)\ \PPB{{\pi_2}^{-1}\gamma_1\gamma_2'}(N).
\end{align*}
	\end{proof}

\subsubsection{Semiclassical asymptotics}
Combining the quantum/classical decomposition of $\tau_{12}^{(N)}$,
the boundedness of $\Quantum_{12}^{(N)}$, and the convergent series
expansion of $\Classical_{12}^{(N)}$, we obtain the following semiclassical
asymptotic expansion of $\tau_{12}^{(N)}$.

	\begin{corollary}
	\label{cor:TwoPointSemiclassical}
		For any sequence $\hbar_N=O(1)$, we have 
			\begin{equation*}
				\tau_{12}^{(N)} = \Classical_{12}^{(N)} + O(\hbar_N).
			\end{equation*}
		In particular, if $\hbar_N = o(N^{-2l})$, then
			\begin{equation*}
				\tau_{12}^{(N)} = \sum_{k=0}^l \frac{e_k^{(12)}(N)}{N^{2k}} + o\left(\frac{1}{N^{2l}}\right).
			\end{equation*}
	\end{corollary}
	
We may now combine \cref{cor:SemiclassicalAsymptotics} and 
\cref{cor:TwoPointSemiclassical} to obtain the 
semiclassical/large-dimension asymptotics of the difference 
	\begin{equation*}
		\tau_{12}^{(N)} - \tau_1^{(N)}\tau_2^{(N)}.
	\end{equation*}
According to \cref{cor:TwoPointSemiclassical}, 
provided $\hbar_N \to 0$ as $N \to \infty$, we have
	\begin{equation*}
		\tau_{12}^{(N)} = \Classical_{12}^{(N)} + o(1)
	\end{equation*}
as $N \to \infty$.  Moreover, by \cref{cor:SemiclassicalAsymptotics},
we have that
	\begin{equation*}
		\tau_1^{(N)}\tau_2^{(N)} = \Classical_1^{(N)}\Classical_2^{(N)} + o(1)
	\end{equation*}
in this same regime, where $\Classical_1^{(N)}$ is the classical part of 
$\tau_1^{(N)}$ and $\Classical_2^{(N)}$ is the classical part of $\tau_2^{(N)}$.
Thus, we have that
	\begin{equation*}
		\tau_{12}^{(N)} - \tau_1^{(N)}\tau_2^{(N)} =  \Classical_{12}^{(N)} -
		\Classical_1^{(N)}\Classical_2^{(N)} + o(1)
	\end{equation*}
in the semiclassical/large-dimension limit.  This means that the asymptotics of 
$\tau_{12}^{(N)} - \tau_1^{(N)}\tau_2^{(N)}$ in the semiclassical/large-dimension
limit coincide with the corresponding classical random matrix asymptotics in the
large-dimension limit, up to replacing the Newton power-sum symmetric functions
with the BPP symmetric functions.  Thus, any computation of the covariance
of the Newton observables $\langle \overline{p}_{k_1}^{(N)}\overline{p}_{k_2}^{(N)} \rangle_c$  
of the classical system \eqref{eqn:ContinuousParticleEnsemble} 
holds verbatim for the computation of the covariance of the BPP observables 
$\langle \p_{k_1}^{(N)}\p_{k_2}^{(N)} \rangle_c$ of the quantum system 
\eqref{eqn:DiscreteParticleEnsemble}.
In particular, either of the methods of \cite{HigherOrderFreeness2} or 
\cite{CollinsMatsumotoNovak2017} (the first
of which is based on the combinatorics of annular noncrossing partitions, whereas the second
uses the combinatorics of monotone walks on symmetric groups)
already developed to estimate the covariance
of traces 
\[ \E[\tr(Z_N^{k_1})\tr(Z_N^{k_2})]-\E[\tr(Z_N^{k_1})]\ \E[\tr(Z_N^{k_2})] \] 
of the classical random Hermitian matrix $Z_N=X_N+Y_N$ in the large $N$ limit
applies verbatim to estimate the covariance of traces 
\[ \E[\tr(C_N^{k_1})\tr(C_N^{k_2})]-
\E[\tr(C_N^{k_1})]\ \E[\tr(C_N^{k_2})]\] 
of the quantum random matrix $C_N=A_N+B_N$.
Either option may be selected to show that 
	\begin{equation}
		\lim_{N \to \infty} \langle \p_k^{(N)}\p_k^{(N)} \rangle_c =0,
	\end{equation}
which, by Chebyshev's inequality, implies Theorem \ref{thm:QuantumLLN}.

\section*{Acknowledgments}

Beno\^\i t Collins was supported by JSPS Kakenhi grants number 26800048 and 15KK0162,  and 
grant number ANR-14-CE25-0003.

Jonathan Novak was supported by a Simons collaboration grant.

Piotr Śniady has been supported by 
\emph{Narodowe Centrum Nauki}, grant number 2014/15/B/ST1/00064.

All three authors acknowledge a fruitful working environment at Herstmonceux Castle during
the Fields Institute meeting on Quantum Groups and Quantum Information in July 2015,
where this project was conceived, and at the Institut Henri Poincar\'e in February 2017 during the 
Combinatorics and Interactions trimester, where it was completed.

\bibliographystyle{alpha}

\bibliography{biblio-minimal}

\appendix

\section{Rudiments of Free Probability}
\label{app:FreeProb}
Here we briefly outline the basic notions from Free Probability Theory 
which are used in the body of the paper.  This is far from a complete
treatment; further references are the 
texts \cite{DNV,MingoSpeicher_book,NicaSpeicher_book,TaoRMT}, 
the lecture notes \cite{Novak,Shlyakhtenko},
and the brief pr\'ecis \cite{NovakSniady}.

A \emph{noncommutative probability space} is a pair $(\A,\tau)$ consisting of
a unital, associative $\C$-algebra $\A$ together with a unital linear functional 
$\tau \colon \A \to \C$ which is assumed to be a trace: $\tau(AB)=\tau(BA)$ for all
$A,B \in \A$.  The elements of $\A$ are to be thought of as complex-valued
random variables on some underlying Kolmogorov triple $(\Omega,\mathcal{F},\P)$,
with $\E$ playing the role of expectation with respect to the probability measure 
$\P$.  Of course, since $\A$ is allowed to be noncommutative, such a triple may 
not exist.  Elements of $\A$ are \emph{quantum random variables}.

Given a random variable $A \in \A$, the \emph{distribution} of $\A$ is the moment
sequence of $A$:
	\begin{equation*}
		\tau(A^p), \quad p \in\N^*.
	\end{equation*}
Given a pair $A,B$ of quantum random variables in $\A$, their \emph{joint
distribution} is the data set
	\begin{equation*}
		\tau(A^{p(1)}B^{q(1)} \cdots A^{p(d)}B^{q(d)}), \quad
		d \in \N^*,\ p,q \colon [d] \to \N
	\end{equation*}	
obtained by evaluating $\tau$ on all words in $A$ and $B$.  These expectations
are called the \emph{mixed moments} of $A$ and $B$.  

In the context of noncommutative probability, it is reasonable to consider $A,B$ to
be \emph{independent} if there is a universal rule for computing their joint distribution
from knowledge of their individual distributions (``universal'' means that this 
rule does not depend on the individual distributions of $A$ and $B$).  
One such rule comes to us from 
classical probability: $A$ and $B$ are said to be \emph{classically independent}
if they commute, and $\tau(A^pB^q)=\tau(A^p)\tau(B^q)$ for any $p,q \in \N^*$.
In this case, one has 
	\begin{equation*}
		\tau(A^{p(1)}B^{q(1)} \cdots A^{p(d)}B^{q(d)}) = \tau(A^{|p|})\tau(B^{|q|}),
	\end{equation*}
for any mixed moment.

A second universal independence rule for quantum random variables, which is 
truly noncommutative in nature, was discovered by Voiculescu \cite{Voiculescu1991}. 
It is modelled on free products and is substantially more complicated than 
classical independence, which is modelled on tensor products.  A pair of random variables
$A,B \in \A$ are \emph{freely independent} if 
	\begin{equation*}
		\tau\big(f_1(A)g_1(B) \cdots f_d(A)g_d(B)\big)=0
	\end{equation*}
whenever $f_1,g_1,\dots,f_d,g_d$ are univariate polynomials such that
	\begin{equation*}
		\tau\big(f_1(A)\big)=\tau\big(g_1(B)\big)= \dots =\tau\big(f_d(A)\big)=\tau\big(g_d(B)\big)=0.
	\end{equation*}
It is a fact that classical independence and free independence are the \emph{only}
universal independence rules for quantum random variables, see \cite{NicaSpeicher_book}.

It is a non-obvious fact that one may express a mixed moment of free random variables
$A,B$ in terms of pure moments, and indeed the explicit rule for doing so is quite complicated.
This rule may be formulated as follows.  For each positive integer 
$d \in \N^*$, and each permutation $\pi \in \Sy{d}$, define a $d$-linear functional 
	\begin{equation*}
		\tau_\pi \colon \underbrace{\A \times \dots \times \A}_{d \text{ factors}}
		\to \C
	\end{equation*}
using the cycle structure of $\pi$ in the natural way.  For example, if $d=6$ and
$\pi=(1\ 4\ 2)(3\ 6)(5)$, then 
	\begin{equation*}
		\tau_\pi(X_1,X_2,X_3,X_4,X_5,X_6) =
		\tau(X_1X_4X_2)\ \tau(X_3X_6)\ \tau(X_5).
	\end{equation*}
Observe that $\tau_\pi$ is well-defined because $\tau$ is a trace.  The mixed moments
of a pair $A,B$ of free random variables decompose into pure moments according to the rule
	\begin{multline*}
		\tau(A^{p(1)}B^{q(1)} \cdots A^{p(d)}B^{q(d)})
		= \\ \sum_{\substack{({\pi_1},{\pi_2}) \in \Sy{d}^2 \\ 
		\genus({\pi_1},{\pi_1}^{-1}{\pi_2},{\pi_2}^{-1}\gamma)=0}}
		(-1)^{|{\pi_1}^{-1}{\pi_2}|}\ \vec{W}_0({\pi_1},{\pi_2})\
		\mathbf{P}_{\pi_1}(A)\ \mathbf{P}_{{\pi_2}^{-1}\gamma}(B),
	\end{multline*}
where
	\begin{equation*}
		\mathbf{P}_{\pi_1}(A) = \tau_{\pi_1}(A^{p(1)},\dots,A^{p(d)}) 
		\quad\text{ and }\quad
		\mathbf{P}_{{\pi_2}^{-1}\gamma}(B)=\tau_{{\pi_2}^{-1}\gamma}(B^{q(1)},\dots,B^{q(d)})
	\end{equation*}
and $(-1)^{|{\pi_1}^{-1}{\pi_2}|}\vec{W}_0({\pi_1},{\pi_2})$ is the leading order of the Weingarten 
function, i.e.~$\vec{W}_0({\pi_1},{\pi_2})$ is
the number of monotone geodesic paths from ${\pi_1}$ to ${\pi_2}$ in $\Sy{d}$, as 
in \eqref{eqn:WeingartenReduced}.  

\section{Corrected proof of \cref{prop:Antiexceedance}}
\label{sec:erratum}

Regrettably, the proof of \cref{prop:Antiexceedance} presented by Biane \cite[Proposition 8.5, part (3)]{Biane98}
is not completely correct. That proof was based on a commutation relation \cite[top formula on page 166]{Biane98}
fulfilled by the entries of the powers of the matrix $Z_N$;
a commutation relation which with our notations would take the form
\begin{equation}
		\label{eqn:HigherCommutation}
		[(Z_N^m)_{ij},(Z_N^n)_{kl}] = 
		\hbar_N \left[ \delta_{jk}\ (Z_N^{m+n-1})_{il} - \delta_{li}\ (Z_N^{m+n-1})_{kj}\right].
\end{equation}
This commutation relation \emph{does not} hold true in general --- in fact, it fails unless
one of the exponents $m,n$ is equal to $1$.
In this Section will provide a correct proof which is based on the ideas of 
Biane \cite[Section 8]{Biane98},
but does not make use of the  false statement \eqref{eqn:HigherCommutation}.

\subsection{Biasimirs and nice Coxeter conjugations}

\begin{definition}
Let $\pi\in\Sy{d}$ be a permutation. 
We say that the passage from $\pi$ to to $\tau \pi \tau^{-1}$ is a 
\emph{nice Coxeter conjugation} if 
$\tau=(j\ \ j+1)$ with $1\leq j<d$ is a Coxeter transposition such that $\pi(j)\neq j+1$.
\end{definition}

\begin{lemma}
\label{lem:bad-things-that-can-happen-while-conjugating}
If $\tau \pi \tau^{-1}$ is a nice Coxeter conjugation of $\pi\in\Sy{d}$ by $\tau=(j\ \ j+1)$ 
then the elements $C_{\pi}^{(\mathbf{1})}$ and $C_{\tau \pi \tau^{-1}}^{(\mathbf{1})}$
are related by the following relations.
\begin{itemize}
   \item If $j$ or $j+1$ is a fixpoint of $\pi$ then
\[ C^{(\mathbf{1})}_{\pi} - C^{(\mathbf{1})}_{\tau \pi\tau^{-1}} =  0.\]

   \item If neither $j$ nor $j+1$ is a fixpoint of $\pi$ and
$\pi(j+1)\neq j$ then there exist permutations $\pi',\pi''\in\Sy{d-1}$
with the property that
\[ C^{(\mathbf{1})}_{\pi} - C^{(\mathbf{1})}_{\tau \pi\tau^{-1}} =  
   \hbar_N\  C^{(\mathbf{1})}_{\pi'}- \hbar_N\ C^{(\mathbf{1})}_{\pi''}.\]
Furthermore, $\aex \pi'\leq \aex \pi$ and $\aex \pi''\leq \aex \pi$

   \item If $\pi(j+1)= j$ there exist permutations $\pi',\pi'''\in\Sy{d-1}$
with the property that
\[ C^{(\mathbf{1})}_{\pi} - C^{(\mathbf{1})}_{\tau \pi\tau^{-1}} =  
   \hbar_N\  C^{(\mathbf{1})}_{\pi'}- \hbar_N N\ C^{(\mathbf{1})}_{\pi'''}.\]
Furthermore, $\aex \pi'\leq \aex \pi$ and $\aex \pi'''+1\leq \aex \pi$.
\end{itemize}

\end{lemma}
\begin{proof}
This result corresponds to a special case of a result of Biane \cite[Lemma 8.3]{Biane98}.
In this specific case his proof is correct up to the point when he says that the remaining part
\emph{``follows by inspection''}. 
However, by performing in detail this inspection we obtained formulas for the 
permutations $\pi'$ and $\pi''$ which differ from the ones of Biane;
we present our findings in the following.

\medskip

\emph{Permutation $\pi'$.} 
The simplest way to describe the permutation $\pi'$ is to view it as a permutation of the set
$[d]\setminus \{ j+1\}$ defined by
\begin{equation}
\label{eq:pi-prim}
 \pi'(k) = \begin{cases} 
                \pi(j+1) & \text{if } k=j, \\
                \pi(j)   & \text{if } \pi(k)=j+1, \\
                \pi(k) & \text{otherwise}. 
             \end{cases}   
\end{equation}
The corresponding Biasimir is defined as a natural modification of \eqref{eqn:Biasimir} given by
\begin{equation}
\label{eq:biasimir-exotic}
		C_{\pi'}^{(\mathbf{1})} := \sum_{i\colon[d]\setminus\{j+1\} \rightarrow [N]}
		(Z_N)_{i(1)i\pi'(1)} \cdots \cancel{(Z_N)_{i(j+1)i\pi'(j+1)}} \dots (Z_N)_{i(d)i\pi'(d)}.
\end{equation}
By relabelling in the order-preserving way the elements of the set $[d]\setminus \{ j+1\}$ 
to the elements of $[d-1]$, the permutation $\pi'$ can be viewed as the usual permutation of $[d-1]$;
in this way \eqref{eq:biasimir-exotic} can be written as a more conventional Biasimir of the form
\eqref{eqn:Biasimir}.

We shall compare now the sets of the antiexceedances of $\pi$ and $\pi'$ (which will be viewed as
in \eqref{eq:pi-prim}).
\begin{itemize}
   \item Each element of $[d] \setminus \{j,j+1, \pi^{-1}(j+1) \}$ is an antiexceedance of $\pi$
if and only if it is an antiexceedance of $\pi'$.

   \item If $j$ is an antiexceedance of $\pi'$ then
\[ j \geq \pi'(j) = \pi(j+1) \]
and, a fortiori, $j+1$ is an antiexceedance of $\pi$.

   \item Suppose that $a:=\pi^{-1}(j+1)$ is an antiexceedance of $\pi'$; we claim that
at least one of the elements $j$ and $a$ is an antiexceedance of $\pi$.
By contradiction, if this would not be the case we would have the following inequalities:
\[ j < \pi(j) 
= \pi'(a)
\leq a
< \pi(a)
=j+1
\]
which would imply that there is some integer number between $j$ and $j+1$ which is clearly not the case.
\end{itemize}
In this way we proved that $\aex \pi'\leq \aex \pi$.

\medskip

\emph{Permutation $\pi'$.} 
Analogously, in the case when $\pi(j+1)\neq j$,
the simplest way to describe the permutation $\pi''$ is to view it as a permutation of the set
$[d]\setminus \{ j\}$ defined by
\begin{equation}
\label{eq:pi-prim-prim}
 \pi''(k) = \begin{cases} 
                \pi(j)     & \text{if } k=j+1, \\
                \pi(j+1)   & \text{if } \pi(k)=j, \\
                \pi(k)     & \text{otherwise}. 
             \end{cases}   
\end{equation}

We shall compare now the sets of the antiexceedances of $\pi$ and $\pi''$ (which will be viewed as
in \eqref{eq:pi-prim-prim}).
\begin{itemize}
   \item Each element of $[d] \setminus \{j,j+1, \pi^{-1}(j) \}$ is an antiexceedance of $\pi$
if and only if it is an antiexceedance of $\pi''$.

   \item If $j+1$ is an antiexceedance of $\pi''$ then
\[ j+1 \geq \pi''(j+1) = \pi(j).\]
Since by assumption $\pi(j)\neq j+1$ it follows that $\pi(j)\leq j$ and
$j$ is an antiexceedance of $\pi$.

   \item Suppose that $a:=\pi^{-1}(j)$ is an antiexceedance of $\pi''$; we claim that
at least one of the elements $j+1$ and $a$ is an antiexceedance of $\pi$.
By contradiction, if this would not be the case we would have the following inequalities:
\[ j+1 < \pi(j+1) 
= \pi''(a)
\leq a
< \pi(a)
=j
\]
which leads to contradiction.
\end{itemize}
In this way we proved that $\aex \pi''\leq \aex \pi$.

\medskip

\emph{Permutation $\pi'''$.} 
Consider now the case $\pi(j+1)=j$,
the simplest way to describe the permutation $\pi'''$ is to view it as a permutation of the set
$[d]\setminus \{ j\}$ defined by
\begin{equation}
\label{eq:pi-prim-prim-prim}
 \pi'''(k) = \begin{cases} 
                \pi(j)     & \text{if } k=j+1, \\
                \pi(k)     & \text{otherwise}. 
             \end{cases}   
\end{equation}

We shall compare now the sets of the antiexceedances of $\pi$ and $\pi'''$.
\begin{itemize}
   \item Each element of $[d] \setminus \{j,j+1\}$ is an antiexceedance of $\pi$
if and only if it is an antiexceedance of $\pi'''$.

   \item If $j+1$ is an antiexceedance of $\pi'''$ then
\[ j+1 \geq \pi'''(j+1) = \pi(j).\]
Since by assumption $\pi(j)\neq j+1$ it follows that $\pi(j)\leq j$ and
$j$ is an antiexceedance of $\pi$.

   \item Additionally $j+1$ is an antiexceedance of $\pi$.
\end{itemize}
In this way we proved that $\aex \pi''' + 1\leq \aex \pi$.

\end{proof}

\subsection{Converting permutations into a canonical form}

\begin{lemma}
\label{lem:aex-is-nice-at-each-step}
Any permutation $\pi$ 
can be transformed into some canonical permutation \eqref{eq:canonical}
by a sequence of nice Coxeter conjugations.
In other words,
there exists a sequence $\tau_1,\dots,\tau_\ell$ with the following two properties:
for  
\[ \pi_i := \tau_i \tau_{i-1} \cdots \tau_1 \pi \tau_1^{-1} \cdots \tau_i^{-1} \]
we have that $\pi_i$ is a nice Coxeter conjugation of $\pi_{i-1}$ for
$1\leq i \leq \ell$
and the final permutation $\pi_\ell$ is of the canonical form \eqref{eq:canonical}.
\end{lemma}
\begin{proof}
Biane \cite[proof of Proposition 8.5 (3)]{Biane98} gives an algorithmic
construction of this sequence of transpositions which we reproduce below in 
a slightly redacted way.

\smallskip

Let $\sigma\in\Sy{d}$ be a permutation and $\tau=(j\ j+1)$ be a Coxeter transposition
with the property that $j$ and $j+1$ do not belong to the same cycle of 
$\sigma$. 
Clearly, $\sigma':= \tau \sigma \tau^{-1}$ is a nice Coxeter conjugation of $\sigma$.
Furthermore,
the set partition of the set $[d]$ which is encodes the cycle decomposition of $\sigma'$
is obtained from the analogous set partition for $\sigma$ by interchanging the roles of
the elements $j$ and $j+1$. It follows that the permutation $\pi$ can be transformed 
by a sequence of nice Coxeter transpositions 
to some permutation $\tilde{\pi}$ with a cycle decomposition given by an 
\emph{interval set partition} of the form
\[ \big\{ \{1,2,\dots,p_1\}, \{p_1+1,p_1+2,\dots,p_2\}, \dots, 
	\{ p_{l-1}+1,   p_{l-1}+2,  \dots, p_l\} \big\}. 
\]
The permutation $\tilde{\pi}$ can be seen as a product of disjoint cycles
thus it is enough to prove Lemma for each of the cycles separately.
Alternatively: it is enough to prove Lemma in the special case when $\pi$ consists of a single cycle.

\smallskip

Assume now that $\pi\in\Sy{d}$ consists of a single cycle. We can assume that $d\geq 2$. 
We denote by $\operatorname{spoil}(\pi)$ the smallest positive integer $i$ for which $\pi(i)\neq i+1$. 
If $\operatorname{spoil}(\pi)=d$ 
then $\pi=(1\ 2\ \dots d)$ is a canonical permutation and there is nothing to prove.
If $m:=\operatorname{spoil}(\pi)<d$ then $\pi(m)-1\geq m+1$ and
the permutation $\pi$ can be transformed by a sequence of Coxeter conjugations to the permutation
\begin{multline}
\label{eq:conjugations-trololo}
\tilde{\pi}:=(m+1\ m+2) \cdots \big(\pi(m)-1\ \ \ \pi(m)\big) \cdot  
 \pi \cdot \\ 
\cdot \big(\pi(m)-1\ \ \ \pi(m)\big)^{-1} 
\cdots  (m+1\ m+2)^{-1}.     
\end{multline}

The first of these conjugations, i.e.
\[\big(\pi(m)-1\ \ \ \pi(m)\big) \cdot  
 \pi \cdot \big(\pi(m)-1\ \ \ \pi(m)\big)^{-1} \]
is a nice Coxeter conjugation. Indeed, if this was not the case then
we would have $\pi\big( \pi(m)-1 \big) = \pi(m)$ and therefore
$\pi(m)-1=\pi(m)$ would lead to contradiction.  
In an analogous way one can show that each of the conjugations in \eqref{eq:conjugations-trololo}
is a nice Coxeter conjugation.

The permutation $\tilde{\pi}$ has the property that 
$\operatorname{spoil}(\tilde{\pi})>\operatorname{spoil}(\pi)$. 
We iterate this procedure on the newly obtained permutation $\tilde{\pi}$; 
it will terminate in a finite time because the statistics
$\operatorname{spoil}$ increases in each step. 

When the procedure terminates, $\pi$ is the canonical cycle.
\end{proof}

\subsection{Proof of \cref{prop:Antiexceedance}}

\begin{proof}[Proof of \cref{prop:Antiexceedance}]

\

\smallskip

\emph{The classical component $\mathbf{P}_\pi^{(r)}$.}
This part of the claim is straightforward bacause the exact form of 
$\mathbf{P}_\pi^{(r)}$ is explicitly known, see the text immediately
after the proof of \cref{prop:BianePolynomials}.

\bigskip

\emph{The quantum component $\mathbf{Q}_\pi^{(r)}$.}
We start by observing that it is enough to consider the special case when
$r=\mathbf{1}$ is a function which identically equal to $1$.
Indeed, in the proof of \cref{prop:BianePolynomials} we have constructed a permutation
$\pi'\in\Sy{|r|}$ with the property that $C_\pi^{(r)}=C_{\pi'}^{(\mathbf{1}_{|r|})}$,
cf.~\eqref{eq:yes-you-can-make-it-flat};
one can easily check that also
$\cyc \pi=\cyc \pi'$ and $\aex \pi=\aex \pi'$.

\medskip

We use induction over $d$ and assume that the result holds true for all permutations
$\pi\in\Sy{d}$.

\smallskip

By \cref{lem:aex-is-nice-at-each-step} the permutation $\pi$ 
can be transformed into some canonical permutation $\hat{\pi}$ of the form \eqref{eq:canonical}
by a sequence of nice Coxeter conjugations.
\cref{lem:bad-things-that-can-happen-while-conjugating} 
applied to each of the conjugations separately
shows that
the difference of the corresponding Biasimirs
\[ C_{\pi}^{(\mathbf{1})} - C_{\hat{\pi}}^{(\mathbf{1})} \]
is a sum of two terms:
\begin{itemize}
   \item a linear combination (with coefficients in $\Z[\hbar_N]$) of the expressions of the form 
\[ C_{\pi'}^{(\mathbf{1})} \]
over some permutations $\pi'\in\Sy{d-1}$ such that $\aex \pi'\leq \aex \pi$, and 

   \item a linear combination (with coefficients in $\Z[\hbar_N]$) of the expressions of the form 
\[ N \ C_{\pi'''}^{(\mathbf{1})} \]
over some permutations $\pi'''\in\Sy{d-1}$ such that $\aex \pi'''+1 \leq \aex \pi$. 
\end{itemize}

We apply the inductive assertion to the above expressions of the form $C_{\sigma}^{(\mathbf{1})}$
for $\sigma\in\Sy{d-1}$ which concludes the proof.
\end{proof}

\section{Proof of \cref{lem:Defect,lem:SecondDefect}}
\label{sec:proof-of-lem:Defect}

The work of Biane \cite[page 173]{Biane98} has some typos 
which impact the correctness of his proof. For this reason we present below a complete proof 
which is based on the ideas of Biane.

\begin{proof}[Proof of \cref{lem:Defect,lem:SecondDefect}]
The first part of \cref{lem:Defect} is obvious, since each cycle of a permutation gives at least one
contribution to the number of antiexceedances.

\medskip

Biane's \cite[Lemma 8.2(1)]{Biane98} in our notations takes the form 
	\begin{equation*}
		\aex (\gamma \tau^{-1})+ \aex( \tau) = d+1 
	\end{equation*}
for an arbitrary $\tau\in\Sy{d}$ while
\cite[Lemma 8.2(2)]{Biane98} implies that
\[ \aex(\sigma) - \aex(\tau) \leq | \sigma^{-1} \tau|\]
for arbitrary $\psi\in\Sy{d}$. The above two relationships can be combined by adding sideways:
\[ \aex (\gamma \tau^{-1}) + \aex(\sigma) \leq d+1 + |\sigma^{-1} \tau|.\]
By setting $\tau:={\pi_1}^{-1} \gamma$, $\sigma:={\pi_2}^{-1} \gamma$ 
we obtain our target inequality \eqref{eq:defects-nice}.

\medskip

The following is the proof of Biane \cite[page 173]{Biane98} but corrected for typos:
\begin{align*}
 \aex (\gamma \epsilon \tau^{-1}) 
&=   d+1 - \aex(\tau \epsilon^{-1}) 
    & \text{by \cite[Lemma 8.2(1)]{Biane98}} \\
&=   d+2 - \aex(\tau) 
    & \text{by \cite[Lemma 8.2(2)]{Biane98}} \\
&\leq  d+2 + |\sigma^{-1} \tau|- \aex(\sigma) 
    & \text{by \cite[Lemma 8.2\textbf{(2)}]{Biane98}}.
\end{align*}

By the substitutions
\[ \tau:= \pi_1^{-1} \gamma \epsilon, \qquad
   \sigma := \pi_2^{-1} \gamma \epsilon.
\]
the above inequality becomes 
\[ \aex(\pi_1) + \aex (\pi_2^{-1} \gamma \epsilon)
  \leq d+2 + |\pi_1^{-1} \pi_2|.
\]
Our target inequality \eqref{eq:defects-nice2}
is now a direct consequence.
\end{proof}

\section{BPP Matrices and Geometric Quantization}
\label{app:GeometricQuantization}

As mentioned in the Introduction, BPP matrices quantize independent
unitarily invariant random Hermitian matrices with deterministic eigenvalues.
This statement falls under the broad umbrella of \emph{geometric quantization}, 
in the sense of Kirillov and Kostant, see e.g.~\cite{KirillovOrbitMethod}.  In this section,
we give a self-contained, physically motivated treatment of this quantization, specific to 
our setting.  The Reader who is not interested in physical arguments
may skip this section entirely.

	\subsection{Toy example}
We begin by considering a toy example: 
a physical system consisting of 
a single stationary particle with an angular momentum. 
For an alternative (but related) exposition of this example see the work of 
Kuperberg \cite{Kuperberg2002}.
We will use the corresponding symmetry group $\group{Spin}(3)\cong \group{SU}(2)$
as a starting point for exploration of the 
unitary group $\group{U}(\ourn)$ and related algebraic and probabilistic objects.

The traditional way to view the angular momentum in Newtonian mechanics
is to regard it as a vector $\vec{J}=(J_x,J_y,J_z)\in\R^3$.
However, for our purposes it will be more convenient to view 
the angular momentum as a \emph{functional on the Lie algebra} of the special orthogonal 
group $\group{SO}(3)$, that is as an element of $\big(\mathfrak{so}(3) \big)^\star$.
This functional $J$ is defined as follows.
For a given $x \in \mathfrak{so}(3)$ we denote by $J(x)$ Noether's invariant corresponding 
to the one-dimensional Lie group $\R\ni t \mapsto e^{t x}\in \group{SO}(3)$ of rotations.
Since the map $x\mapsto J(x)$ is linear, it defines an element of the dual space.

From a conceptual point of view, regarding angular momentum 
as an element of $\big(\mathfrak{so}(3) \big)^\star$
is advantageous; for example it scales nicely to other choices of the dimension
of the physical space than $3$. 
Unfortunately, the mathematical vocabulary
concerning this dual space is rather limited, and hence it will be convenient to 
have a more concrete alternative available. For this reason 
in the following we shall describe the dual of $\mathfrak{so}(3)\cong \mathfrak{su}(2)$ 
in more detail.

\subsection{The dual space}
\label{sec:dual-space}

In greater generality, we are interested in the dual of the Lie algebra
$\mathfrak{su}(\ourn)$ of traceless antihermitian matrices, 
as well as the dual of the Lie algebra $\mathfrak{u}(\ourn)$ of general antihermitian matrices.

Each of these Lie algebras can be equipped with the symmetric,
non-degenerate, bilinear form
\begin{equation}
\label{eq:bilinear-form}   
 \langle x,y \rangle=\Tr x^T y.
\end{equation}
In this way $\big( \mathfrak{su}(\ourn) \big)^{\ast} \cong \mathfrak{su}(\ourn)$
and
$\big( \mathfrak{u}(\ourn) \big)^{\ast} \cong \mathfrak{u}(\ourn)$.
Thanks to these isomorphisms, it makes sense to speak about the 
\emph{eigenvalues} of elements of the dual spaces 
$\big( \mathfrak{su}(\ourn) \big)^{\ast}$
and $\big( \mathfrak{u}(\ourn) \big)^{\ast}$.

In the latter case,
this isomorphism takes the following more concrete form.
Since the complexification $\uu(\ourn) \otimes_\R \C=
\gl(\ourn) = \Mat_N(\C)$ has a matrix structure, it follows that
$\uu(\ourn)\gwia \otimes_\R \C\cong \uu(\ourn) \otimes_\R \C = \Mat_N(\C)$
can be also identified with matrices. 
More specifically, a functional
$x\in \uu(\ourn)\gwia \otimes_\R \C$ corresponds to the matrix
\begin{equation}
\label{eq:strangematrix}
\begin{bmatrix}
x(e_{11}) & \dots  & x(e_{\ourn 1}) \\
\vdots       & \ddots & \vdots       \\
x(e_{1\ourn}) & \dots  & x(e_{\ourn\ourn}) 
\end{bmatrix} = \sum_{k,l} x(e_{kl})\ e_{kl} \in \Mat_N(\C),\
\end{equation}
where $e_{kl}\in \Mat_N(\C) = \uu(\ourn) \otimes_\R \C$ are the standard matrix units.
Indeed, the above matrix defines via \eqref{eq:bilinear-form} a functional which on a matrix unit $e_{ij}$ 
takes the same value as the functional $x$.

Note the subtlety in the formulation of \eqref{eq:strangematrix}:
since $e_{kl}$ is not an antihermitian matrix, 
for $x\in\uu(\ourn)\gwia$ the quantity $x(e_{kl})$
might be not well defined.
Nevertheless, $x(e_{kl})$ may be defined thanks to the observation that
$e_{kl}\in \Mat_N(\C)= \uu(\ourn) \otimes_\R \C$ belongs to the 
\emph{complexification of antihermitian matrices}, thus
we may extend the domain of $x$ by linearity as follows:
\[ x(e_{kl})= x\left( i \frac{e_{kl}+e_{lk}}{2i} + \frac{e_{kl}-e_{lk}}{2} \right):=
i    x\left( \frac{e_{kl}+e_{lk}}{2i}\right)+
x\left( \frac{e_{kl}-e_{lk}}{2} \right).
\]

\subsection{Back to the angular momentum}
Suppose that for some physical Newtonian system its angular momentum 
--- viewed as a vector $\vec{J}\in\R^3$ ---  
is random, with the uniform distribution on the sphere with radius $|J|$.
One can show that this corresponds to $J$ being 
\emph{a random element of the dual space
$\big( \mathfrak{su}(2) \big)^{\ast}$, uniformly random on the manifold
of antihermitian matrices with specified eigenvalues} 
$\pm i\ |J|$. 

In other words, under the isomorphism from \cref{sec:dual-space} the distribution of the
angular momentum coincides with the distribution of the random matrix
\begin{equation}
\label{eq:angular-momentum-as-a-matrix}
U \begin{bmatrix} i \ |J| &               \\
                               & -i \ |J|    \\
     \end{bmatrix}
U^{-1},
\end{equation}
where $U\in \group{SU}(2)$ is a random matrix from the special unitary group, 
distributed according to the Haar measure. 
We now describe a quantum analogue of this probability distribution.

\subsection{Angular momentum in quantum mechanics}
We consider the following quantum analogue of the Newtonian system considered above:
a quantum particle with fixed spin $j\hslash$, where 
$j\in\left\{ 0, \nicefrac{1}{2}, 1 ,\nicefrac{3}{2}, \dots \right\}$
and $\hslash$ denotes the Planck constant.
Such a particle is described by a Hilbert space $\V$,
this space being the appropriate unitary representation 
${\pi_1}\colon \group{Spin}(3)\to \group{GL}(\V)$.
The Lie group $\group{Spin}(3)\cong \group{SU}(2)$ is
the universal cover of the group $\group{SO}(3)$ describing rotations of the physical space.
To be more specific, ${\pi_1}$ is the irreducible representation
of the Lie group $\group{SU}(2)$ with the dimension $2j+1\in\{1,2,\dots\}$. 

In order to sustain the concordance with the Newtonian situation discussed above,
the angular momentum should be a functional
\[ J \colon  \mathfrak{so}(3) \to \End \V \]
which to an element of the Lie algebra 
$x\in\mathfrak{so}(3)$
associates the infinitesimal \emph{Hermitian} generator of the action of the one-parameter Lie group 
$\R \ni t\mapsto e^{t x}\in\group{Spin}(3)$ on its representation $\V$, i.e.
\[ {\pi_1}\left( e^{t x} \right) = e^{-i t \frac{J(x)}{\hslash} }.
\]
The choice of normalization on the right hand side comes from the notations used
in quantum mechanics.
Clearly, this means that (up to a scalar multiple) the angular momentum
\[ -i\frac{J}{\hslash} ={\pi_1} \colon  \mathfrak{so}(3) \to \End\V \]
is a representation of the Lie algebra $\mathfrak{so}(3)=\mathfrak{su}(2)$.
If $\End \V$ is viewed as an algebra of noncommutative random variables,
\begin{equation*}
  -i\frac{J}{\hslash} = {\pi_1} \in \big( \mathfrak{so}(3) \big)^{\ast} \otimes  \End \V  
\end{equation*}
becomes a \emph{quantum random element of the dual space 
$\big( \mathfrak{so}(3) \big)^{\ast}=\big( \mathfrak{su}(2) \big)^{\ast}$}.

Just as before we assume that we have no further information about the particle; 
in other words, the quantum system is in the maximally mixed state 
and thus the algebra $\End \V$ of noncommutative random variables 
is equipped with the state $\tr_\V$.
Just as before, it is convenient to have a concrete matrix representation 
from \cref{sec:dual-space}
for the elements of the dual space $\big( \mathfrak{su}(2) \big)^{\ast}$. 
We shall discuss this concrete representation now.

\subsection{The dual space}
Consider a slightly more general situation in which 
${\pi_1}\colon \mathfrak{u}(\ourn)\to \End\V$ is a representation of the Lie algebra
$\mathfrak{u}(\ourn)$.

Equation \eqref{eq:strangematrix} shows that $-i\frac{J}{\hslash} ={\pi_1}$ can be identified with the matrix
\begin{equation}
\label{eq:reps-as-a-matrix}
  -i\frac{J}{\hslash} = {\pi_1}= \begin{bmatrix}
{\pi_1}(e_{11}) & \dots  & {\pi_1}(e_{\ourn 1}) \\
\vdots       & \ddots & \vdots       \\
{\pi_1}(e_{1\ourn}) & \dots  & {\pi_1}(e_{\ourn\ourn}) 
\end{bmatrix}.
\end{equation}

\subsection{Conclusion}
The above considerations show that from a physicist's point of view, 
for $N=2$ the $2\times 2$ matrix \eqref{eq:reps-as-a-matrix} 
is a natural quantization
of the random matrix \eqref{eq:angular-momentum-as-a-matrix}
which describes the angular momentum in Newtonian mechanics.

It is time to detach from the physical toy example related to the group 
$\group{Spin}(3)\cong \group{SU}(2)$ and consider the general situation
treated in this article. 
The \emph{classical object} which we considered in this section was
a random element of $\big(\mathfrak{u}(\ourn)\big)^{\ast}$
(or, a random antihermitian matrix), sampled
uniformly from the elements with specified spectrum.
Its \emph{quantization} is a BPP matrix:
a quantum random
element of $\big(\mathfrak{u}(\ourn)\big)^{\ast}$
which corresponds to a specified irreducible representation of $\group{U}(\ourn)$.

\subsection{Choice of the matrix structure on $\uu(\ourn)\gwia$}
\label{subsec:matrix-structure}
Unlike in the case of the Lie algebra $\uu(\ourn)$, there is no
canonical choice of matrix structure on the dual $\uu(\ourn)\gwia$. In
\cref{sec:dual-space} this structure was
chosen based on the bilinear form $\langle A,B\rangle=\Tr A^T B$. One can argue
however, that the bilinear form $\langle A,B\rangle=\Tr A B$ would be equally
natural. 
With respect to this new convention,
the representation ${\pi_1}$ viewed as a matrix becomes
\begin{equation} 
\label{eq:reps-as-a-matrixA}
\begin{bmatrix}
{\pi_1}(e_{11}) & \dots  & {\pi_1}(e_{1 \ourn}) \\
\vdots       & \ddots & \vdots       \\
{\pi_1}(e_{\ourn 1}) & \dots  & {\pi_1}(e_{\ourn\ourn}) 
\end{bmatrix} \in \Mat_N(\End \V)=\Mat_N(\C)\otimes\End \V,
\end{equation}
which 
is \emph{not} a BPP matrix.
The matrices \eqref{eq:reps-as-a-matrix} and \eqref{eq:reps-as-a-matrixA} differ
only by transposition with respect to the first factor of the tensor product
$\Mat_N(\C)\otimes\End \V$, an operation known as 
\emph{partial transposition}.
The minor advantage of the notation \eqref{eq:reps-as-a-matrix} is that it
coincides with the notation of \v{Z}elobenko \cite{Zelobenko73} 
who calculated the spectral measure of BPP matrices.

There are, however, no serious advantages of one notation over the other,
since the calculation of the spectral measure of \eqref{eq:reps-as-a-matrixA} can be
done by the analogous methods to those of \v{Z}elobenko \cite{Zelobenko73}.
The only difference is that instead of considering the tensor product with the
canonical representation, one should consider the tensor product with the
contragradient one.

\subsection{Another erratum to \cite{Biane98}}
The existence of two natural choices for the matrix associated to a representation, the 
BPP matrix \eqref{eq:reps-as-a-matrix} and its partial transpose \eqref{eq:reps-as-a-matrixA}, has 
been a source of confusion in the literature,
in particular in the work of Biane \cite{Biane98}. We shall review and clarify this issue here.

On page 163 of \cite{Biane98}, Biane defines an operator $X$ which in our notation
corresponds to \eqref{eq:reps-as-a-matrixA}, the \emph{partial transpose} of BPP operator.  
This is beneficial for his purposes because the spectrum of such transpose BPP matrix is 
asymptotically related to the \emph{transition measure} of the 
Young diagram \cite[Proposition 7.2]{Biane98}. 

The operator $X$ seems to disappear in \cite[Section 8]{Biane98}, 
at which point the closely related \emph{Casimir operators} 
\[ C_\sigma=\sum_{i_1,\dots,i_q} \zeta_{i_1 i_{\sigma(1)}} \cdots  \zeta_{i_q i_{\sigma(q)}}.\]
are introduced and studied. Informally speaking, this Casimir operator is obtained by
selecting certain entries from the matrix we are interested in, 
according to some equalities between the indices.
From the viewpoint of Biane's goals, this choice of the definition of 
Casimir operators is somewhat surprising
because it refers to selecting entries from the BPP matrix \eqref{eq:reps-as-a-matrix}.
With the operator $X$ in mind, it would make more sense to consider the Casimir operators defined by
\[ \widetilde{C}_\sigma=\sum_{i_1,\dots,i_q} \zeta_{i_{\sigma(1)} i_1} \cdots  \zeta_{i_{\sigma(q)} i_q}.\]

And indeed: in \cite[Section 8]{Biane98} a special role is played by Casimir operators 
in the special case when $\sigma=(1,2,\dots,q)$ 
is the full \emph{forward} cycle,
in the which case $C_\sigma=C_q$ corresponds to the partial trace
of the $q$-th power of BPP matrix. However, 
for Biane's purposes it would be more convenient to pay special attention
to the case when $\sigma'=(q,q-1,\dots,2,1)$ is the full \emph{backward} cycle, since
$C_{\sigma'}$ corresponds to the trace of Biane's operator $X$.
An even better solution would be to consider the modified Casimir operator
$\widetilde{C}_{\sigma}$ for $\sigma=(1,2,\dots,q)$ being the full \emph{forward} cycle.
 
The operator $X$ resurfaces in \cite[Section 9.2]{Biane98} where the inevitable notation 
clash occurs on page 172, in which the third displayed equation would be true if Biane's Casimir 
operators were replaced by our modified Casimir operators $\widetilde{C}$.
\end{document}